\documentclass[10pt,a4paper]{amsart}
\usepackage{amsmath,amsthm,amssymb,amsfonts,mathabx}
\usepackage{abraces}
\usepackage{tikz-cd}
\usepackage{thmtools,xcolor}
\usepackage[bookmarks=true,hyperindex,pdftex,colorlinks,citecolor=blue,linkcolor=blue,urlcolor=blue]{hyperref}

\usepackage[english]{babel}
\usepackage{enumitem}

\usepackage{graphicx}

\declaretheoremstyle[
headfont=\color{blue}\normalfont\bfseries,
bodyfont=\color{blue}\normalfont\itshape,
]{colored}

\theoremstyle{plain}
\newtheorem{theorem}{Theorem}[section]
\newtheorem{cor}[theorem]{Corollary}
\newtheorem{corollary}[theorem]{Corollary}
\newtheorem{prop}[theorem]{Proposition}
\newtheorem{lemma}[theorem]{Lemma}

\theoremstyle{definition}

\newtheorem{example}[theorem]{Example}
\newtheorem{examples}[theorem]{Examples}
\newtheorem{rem}[theorem]{Remark}
\newtheorem{remark}[theorem]{Remark}
\newtheorem{definition}[theorem]{Definition}
\newtheorem{problem}[theorem]{Problem}

\newcommand{\R}{\mathbb{R}}
\newcommand{\N}{\mathbb{N}}
\newcommand{\C}{\mathbb{C}}
\newcommand{\KK}{\mathbb{K}}
\newcommand{\T}{\mathbb{T}}

\newcommand{\W}{\mathcal{W}}
\newcommand{\Lin}{\mathcal{L}}
\newcommand{\K}{\mathcal{K}}
\newcommand{\F}{\mathcal{F}}
\newcommand{\eps}{\varepsilon}

\newcommand{\vertiii}[1]{{\left\vert\kern-0.25ex\left\vert\kern-0.25ex\left\vert #1
 \right\vert\kern-0.25ex\right\vert\kern-0.25ex\right\vert}}

\DeclareMathOperator{\dist}{dist}

\DeclareMathOperator{\re}{Re}
\DeclareMathOperator{\spann}{span}

\DeclareMathOperator{\co}{co}

\DeclareMathOperator{\NA}{NA}

\DeclareMathOperator{\Id}{Id}
\DeclareMathOperator{\QNA}{QNA}

\DeclareMathOperator{\LipA}{LipA}

\newcommand{\Lip}{{\mathrm{Lip}}_0}

\renewcommand{\leq}{\leqslant}
\renewcommand{\geq}{\geqslant}

\title[On quasi norm attaining operators between Banach spaces]{On quasi norm attaining operators between Banach spaces}

\author[G.~Choi]{Geunsu Choi}
\address[G.~Choi]{Department of Mathematics, Institute for Industrial and Applied Mathematics, Chungbuk National University, Cheongju, Chungbuk 28644, Republic of Korea}
\email{\texttt{chlrmstn90@gmail.com}}

\author[Y.~S.~Choi]{Yun Sung Choi}
\address[Y.~S.~Choi]{Department of Mathematics, POSTECH, Pohang 790-784, Republic of Korea}
\email{\texttt{mathchoi@postech.ac.kr}}

\author[M.~Jung]{Mingu Jung}
\address[Mingu Jung]{Department of Mathematics, POSTECH, Pohang 790-784, Republic of Korea \newline
	\href{http://orcid.org/0000-0003-2240-2855}{ORCID: \texttt{0000-0003-2240-2855} }}
\email{\texttt{jmingoo@postech.ac.kr}}

\author[M.~Mart\'in]{Miguel Mart\'in}
\address[Mart\'{\i}n]{Universidad de Granada, Facultad de Ciencias.
Departamento de An\'{a}lisis Matem\'{a}tico, 18071-Granada,
Spain \newline
	\href{http://orcid.org/0000-0003-4502-798X}{ORCID: \texttt{0000-0003-4502-798X} }}
\email[Mart\'{\i}n]{mmartins@ugr.es}

\thanks{The first author was supported by the National Research Foundation of Korea (NRF) grant funded by the Korea government (MSIT) [NRF-2020R1C1C1A01012267]. The second and third authors were supported by the Basic Science Research Program through the National Research Foundation of Korea (NRF) funded by the Ministry of Education (NRF-2019R1A2C1003857). The fourth author has been partially supported by projects PGC2018-093794-B-I00 (MCIU/AEI/FEDER, UE) and FQM-185 (Junta de Andaluc\'{\i}a/FEDER, UE)}

\date{April 23nd, 2020}


\begin{document}
	
\begin{abstract}
	We introduce a weakened notion of norm attainment for bounded linear operators between Banach spaces which we call \emph{quasi norm attaining operators}. An operator $T\colon X \longrightarrow Y$ between the Banach spaces $X$ and $Y$ is quasi norm attaining if there is a sequence $(x_n)$ of norm one elements in $X$ such that $(Tx_n)$ converges to some $u\in Y$ with $\|u\|=\|T\|$. Norm attaining operators in the usual sense (i.e.\ operators for which there is a point in the unit ball where the norm of its image equals the norm of the operator) and compact operators satisfy this definition. The main result of the paper is that strong Radon-Nikod\'{y}m operators such as weakly compact operators can be approximated by quasi norm attaining operators (even by a stronger version of the definition), which does not hold for norm attaining operators. This allows us to give characterizations of the Radon-Nikod\'{y}m property in term of the denseness of quasi norm attaining operators for both domain spaces and range spaces, extending previous results by Bourgain and Huff. We also present positive and negative results on the denseness of quasi norm attaining operators, characterize both finite dimensionality and reflexivity in terms of quasi norm attaining operators, discuss conditions to obtain that quasi norm attaining operators are actually norm attaining, study the relationship with the norm attainment of the adjoint operator, and present some stability properties. We finish the paper with some open questions.
\end{abstract}

\maketitle

\section{Introduction}
Let $X$ and $Y$ be Banach spaces over the field $\KK$ (which will always be $\R$ or $\C$). We write $B_X$ and $S_X$ to denote the unit ball and the unit sphere of $X$, respectively. By $\Lin(X,Y)$ we mean the Banach space of all bounded linear operators from $X$ to $Y$ endowed with the operator norm; we just write $X^*=\Lin(X,\KK)$ for the dual space of $X$. The space of compact (respectively, weakly compact) operators from $X$ to $Y$ will be denoted by $\K(X,Y)$ (respectively, $\W(X,Y)$). We denote by $\mathbb{T}$ the unit sphere of the base field $\KK$ and use $\re(\cdot)$ to denote the real part, being just the identity when dealing with real scalars.

Recall that $T\in \Lin (X,Y)$ \emph{attains its norm} ($T\in \NA(X,Y)$ in short) if there is a point $x_0\in S_X$ such that $\|Tx_0\|=\|T\|$; in this case, we say that $T$ \emph{attains its norm at} $x_0$. Equivalently, $T\in \NA(X,Y)$ if and only if $T(B_X)\cap \|T\|S_Y \neq \emptyset$. The study of norm attaining operators goes back to the 1963's paper \cite{Lindenstrauss} by J.~Lindenstrauss, who first discussed the possible extension to the setting of general operators of the famous Bishop-Phelps' theorem on the denseness of norm attaining functionals, that is, to study when the set $\NA(X,Y)$ is dense in $\Lin(X,Y)$. He showed that this is not always the case and provided several positive conditions. We refer the reader to the expository paper \cite{Aco-survey} by M.~D.~Acosta for a detailed account on the main results on this topic. Let us just mention that important contributions to this topic have been done by, among many other authors, J.~Bourgain, R.~Huff, C.~Stegall, and V.~Zizler in the 1970's, J.~Partington and W.~Schachermayer in the 1980's, R.~Aron, M.~D.~Acosta, W.~T.~Gowers, and R.~Pay\'{a} in the 1990's, and that it is still an active area of research by many authors, especially in the related topic of the study of the Bishop-Phelps-Bollob\'{a}s property introduced in 2008 \cite{BPBp}. We refer to \cite{Acosta-survey-BPBP,M-RACSAM} for an account of the most recent results.

Our aim in this paper is to study the following concept, which is a weakening of norm attainment.

\begin{definition}\label{def:QNA}
We say that a bounded linear operator $T \in \Lin (X,Y)$ \emph{quasi attains its norm} (in short, $T \in \QNA(X,Y)$) if $\overline{T(B_X)} \cap \|T\| S_Y \neq \emptyset$. Equivalently, $T \in \QNA(X,Y)$ if and only if there exists a sequence $(x_n) \subseteq S_X$ such that $(Tx_n)$ converges to some vector $u \in Y$ with $\|u\|=\|T\|$; in this case, we say that $T$ \emph{quasi attains its norm toward $u$}.
\end{definition}

The main goal here will be to study when the set $\QNA(X,Y)$ is dense in the space $\Lin(X,Y)$, providing both negative and positive results, and applications.

The concept presented in Definition~\ref{def:QNA} appeared previously in a paper by G.~Godefroy \cite{G} of 2015 in the more general setting of Lipschitz maps, as follows. We write $\Lip(X,Y)$ for the real Banach space of all Lipschitz maps from a Banach space $X$ to a Banach space $Y$ vanishing at $0$, endowed with the Lipschitz number $\|\cdot\|_\text{Lip}$ as a norm.

\begin{definition}\label{def:LipA}
Let $X$ and $Y$ be real Banach spaces. A Lipschitz map $f \in \Lip (X,Y)$ \emph{attains its norm toward} $u \in Y$ (in short, $f \in \LipA(X,Y)$) if there exists a sequence of pairs $\bigl((x_n,y_n)\bigr) \subseteq \widetilde{X}$ such that
	\[
	\frac{f(x_n)-f(y_n)}{\|x_n-y_n\|} \longrightarrow u \quad \text{with } \|u\|=\|f\|_\text{Lip},
	\]
where $\widetilde{X}=\{(x,y) \in X^2 \colon x \neq y \}$.
\end{definition}

Observe that it is clear from the definitions that
\begin{equation}\label{eq:qna=lincaplipa}
\QNA(X,Y) = \Lin (X,Y) \cap \LipA(X,Y)
\end{equation}
for all Banach spaces $X$ and $Y$. It is shown in \cite{G} that no Lipschitz isomorphism from $c_0$ to any renorming $Z$ of $c_0$ with the Kadec-Klee property belongs to $\LipA(c_0,Z)$. This complements an old result by J.~Lindenstrauss \cite{Lindenstrauss} concerning linear operators. In the words of Godefroy, the example shows that even the greater flexibility allowed by non linearity (and the weakening of the new definition of norm attainment) does not always provide norm attaining objects. From this example, we deduce that $\LipA(c_0,Z)$ is not dense in $\Lip(c_0,Z)$ (see \cite[Example~3.6]{CCM} for the details). Besides, from \eqref{eq:qna=lincaplipa}, no linear isomorphism from $c_0$ onto $Z$ belongs to $\QNA(c_0,Z)$, and being the set of linear isomorphisms open in $\Lin(c_0,Z)$, this gives:

\begin{example}\label{example1:QNA_not_dense_G}
{\slshape If $Z$ is a renorming of $c_0$ with the Kadec-Klee property, then $\QNA(c_0,Z)$ is not dense in $\Lin(c_0,Z)$.}
\end{example}

Observe that this result shows that denseness of quasi norm attainment is not a trivial property. On the other hand, going to positive results, the following easy remarks will provide the first ones.

\begin{remark}\label{basic_remark}
{\slshape Let $X$ and $Y$ be Banach spaces. Then, we have:}
\begin{enumerate}
\item[\textup{(a)}] $\NA (X,Y) \subseteq \QNA(X,Y).$
\item[\textup{(b)}] $\K (X,Y) \subseteq \QNA (X,Y).$
\end{enumerate}
\end{remark}

\begin{proof}
(a) is obvious; if $T$ attains its norm at $x_0\in S_X$, then $T$ quasi attains its norm toward $T(x_0)$. To show the validity of (b), consider $T \in \K(X,Y)$ and take $(x_n) \subseteq B_X$ such that $\|Tx_n\| \longrightarrow \|T\|$. Since $T(B_X)$ is relatively compact, passing to a subsequence, we may assume that $Tx_n \longrightarrow y_0$ for some $y_0 \in \overline{T(B_X)}$. Clearly, $\|y_0\|=\|T\|$, hence $T \in \QNA(X,Y)$.
\end{proof}

Now, if $\NA(X,Y)$ is dense in $\Lin(X,Y)$ for a pair of Banach spaces, then even the more $\QNA(X,Y)$ is dense in $\Lin(X,Y)$. This happens, for instance, when $X$ is reflexive (Lindenstrauss) or, more generally, when $X$ has the Radon-Nikod\'{y}m property, RNP in short (Bourgain), or when $Y$ has a geometric property called property $\beta$ (Lindenstrauss) which is accessible by equivalent renorming to every Banach space (Schachermayer). This is also the case for some concrete pairs of Banach spaces as $(L_1(\mu),L_1(\nu))$, $(L_1(\mu),L_\infty(\nu))$, $(C(K_1),C(K_2))$, $(C(K),L_1(\mu))$, among many others. A detailed account of all these examples can be found in Examples~\ref{examples:NA-dense}. Similarly, if all operators from a Banach space $X$ to a Banach space $Y$ are compact, then $\QNA(X,Y)=\Lin(X,Y)$ by Remark~\ref{basic_remark}.(b), so this produce a long list of examples of pairs by using Pitt's Theorem or results by Rosenthal. An extensive list of examples and of references where one may found much more examples of this kind appears in Examples~\ref{examples:K(X,Y)=L(X,Y)}. This allows to present an example of pair of Banach spaces $(X,Y)$ such that $\QNA(X,Y)$ is dense in $\Lin(X,Y)$, while $\NA(X,Y)$ is not (see Example~\ref{example:QNA-butNA-not-first}).

About negative results, we will show also in Section~\ref{sect:first_examples} that some known examples of pairs of Banach spaces $(X,Y)$ for which $\NA(X,Y)$ is not dense in $\Lin(X,Y)$ actually satisfy the stronger result that $\QNA(X,Y)$ is not dense in $\Lin(X,Y)$. This is the case of $(c_0,Y)$ when $Y$ is any strictly convex renorming of $c_0$ (see Example \ref{example2:QNA_not_dense_G}), complementing Example~\ref{example1:QNA_not_dense_G}, or $(L_1[0,1],C(S))$ for a Hausdorff compact space $S$ constructed in \cite{JJ}. The idea for all these negative results is that quasi norm attainment and norm attainment are equivalent for monomorphisms (Lemma~\ref{lem_paya}), which is a result suggested to us by R.~Pay\'{a}. This idea also allows us to deduce from an old result of R.~Huff \cite{H} that if a Banach space $X$ does not have the RNP, then there are equivalent renorming $X_1$ and $X_2$ of $X$ such that $\QNA(X_1,X_2)$ is not dense in $\Lin(X_1,X_2)$ (see Proposition~\ref{prop:huff}). Therefore, a Banach space $X$ has the RNP if (and only if) $\QNA(X',Y)$ is dense in $\Lin(X',Y)$ for every equivalent renorming $X'$ of $X$ and every Banach space $Y$, giving a stronger result than the analogous one for norm attaining operators. If we focus on range spaces, it also follows from Proposition~\ref{prop:huff} that a Banach space $Y$ has the RNP provided that for every Banach space $X$ and every renorming $Y'$ of $Y$, the set $\QNA(X,Y')$ is dense in $\Lin(X,Y')$. Aiming at a reciprocal result to the previous one, we cannot lay hold on a result for norm attaining operators in the classical sense: a celebrated result of T.~Gowers \cite[Appendix]{G2} shows that there is a Banach space $X$ satisfying that $\NA(X,\ell_p)$ is not dense in $\Lin(X,\ell_p)$ for $1 < p < \infty$. Here is where the differences between the denseness of norm attaining operators and the denseness of quasi norm attaining operators are more clear. We devote Section~\ref{section:RNP-sufficient} to present the main result of the paper (Theorem~\ref{theo:RNP-QNA}): strong RNP operators can be approximated by operators which quasi attain their norm in a strong sense (uniquely quasi norm attaining operators, see Definition~\ref{definition:strong-sense}). As a consequence, we obtain that $\QNA(X,Y)$ is dense in $\Lin(X,Y)$ if $Y$ has the RNP, obtaining thus a characterization of the RNP in terms of the range space. So, with quasi norm attaining operators we obtain characterizations of the RNP which are symmetric on the domain and range spaces (see Corollary~\ref{corollary:RNP-equivalence}). This also allows us to present examples of pairs $(X,Y)$ such that $\QNA(X,Y)$ is dense, $\NA(X,Y)$ is not, and $\K(X,Y)\neq \Lin(X,Y)$ (Examples~\ref{examples:QNAdense-NAnot}). The particular cases of Theorem~\ref{theo:RNP-QNA} for compact or weakly compact operators are actually interesting (Corollary~\ref{coro:XRNPYRNP-QNAdense-strongversion}).

The already quoted Theorem~\ref{theo:RNP-QNA} has also consequences of different type: other kind of applications can be given for Lipschitz maps, multilinear maps, and $n$-homogeneous polynomials. For instance, $\LipA(X,Y)$ is dense in $\Lip(X,Y)$ for every Banach space $X$ and every Banach space $Y$ with the RNP as stated in Corollary~\ref{corollary:RNP-LDA-dense}. On the other hand, there is a natural definition of quasi norm attaining multilinear maps (see Definition \ref{definition:QNA-multilinear}) or quasi norm attaining $n$-homogeneous polynomials (Definition \ref{definition:QNA-polynomial}) for which the RNP of the range space is a sufficient condition, see Corollaries \ref{corollary:multilinearMaps} and \ref{corollary:polynomials}, respectively.

We study in Section~\ref{section:reflexivity-finite-dimension} when each of the inclusions in the chain $\NA(X,Y)\subset \QNA(X,Y)\subset \Lin(X,Y)$ can be an equality. In the first case, we show that the equality $\NA(X,Y)= \QNA(X,Y)$ characterizes the reflexivity of $X$ if its holds for a nontrivial space $Y$ (and then holds for all Banach spaces $Y$), see Proposition~\ref{prop:reflexive}. On the other hand, we show that if either $X$ or $Y$ is infinite dimensional, then $\QNA(X,Y)\neq \Lin(X,Y)$ (Corollary~\ref{corollary:finite-dimension}), and we show that this provides a result related to remotality (Corollary~\ref{coro:remotality}).

In Section~\ref{sect:further}, we first study the relationship between the quasi norm attainment and the norm attainment of the adjoint operator. If $T\in \QNA(X,Y)$, then $T^*\in \NA(Y^*,X^*)$ (Proposition~\ref{prop:QNAimpliesNAadjoint}), but the reciprocal result is not true (a concrete example is given in Example~\ref{eample:adjoint-NA-operator-not-QNA}). However, the equivalence holds for weakly compact operators (Proposition~\ref{prop_W_QNA}). Secondly, we study
possible extensions of Lemma~\ref{lem_paya} on conditions assuring that quasi norm attainment implies (classical) norm attainment. We show that quasi norm attaining operator with closed range and proximinal kernel are actually norm attaining (Proposition~\ref{prop_proximinal}), and the same is true if the annihilator of the kernel of the operator is contained in the set of norm attaining functionals on the space (Proposition~\ref{prop:QNA=NA_lineability_NA}). We will see that some of those conditions cannot be removed from the assumption by presenting specific examples; for instance, Example~\ref{example:Gowers-injective-QNA-no-NA} reveals that there exists a quasi norm attaining injective weakly compact operator which does not even belong to the closure of the set of norm attaining operators.

We discuss in Section~\ref{section:stability} some stability properties for the denseness of quasi norm attaining operators, which can be obtained analogously from the ones on norm attaining operators. Nevertheless, we have been not able to transfer to this new setting all the known stability results for norm attaining operators in the classical setting.

Finally, we devote Section~\ref{section:Remarks-open-questions} to present some remarks and open questions on the subject.

\section{Denseness of the set of quasi norm attaining operators: first positive and negative examples}\label{sect:first_examples}

As the first part of this section, we would like to present examples of pairs of Banach spaces $X$ and $Y$ such that $\QNA(X,Y)$ is dense in $\Lin(X,Y)$ using Remark~\ref{basic_remark}: those such that $\NA(X,Y)$ is dense in $\Lin(X,Y)$ and those satisfying that $\K(X,Y)=\Lin(X,Y)$.

With respect to the first condition, the following list of examples summarizes many of the known results of this kind. Instead of giving a long list of definitions and references, we send the interested reader to the survey paper \cite{Aco-survey} of 2006 for this. For results which do not appear there (as they are newer), we include the reference.

\begin{examples}\label{examples:NA-dense}\slshape
Any of the following conditions assures that $\NA(X,Y)$ is dense in $\Lin(X,Y)$ and so, that $\QNA(X,Y)$ is dense in $\Lin(X,Y)$:
	\begin{enumerate}
	\item[\textup{(a)}] \textup{(}Bourgain\textup{)} $X$ has the Radon-Nikod\'{y}m property; in particular, if $X$ is reflexive.
	\item[\textup{(b)}] \textup{(}Schachermayer\textup{)} $X$ has property $\alpha$ \textup{(}for instance, if $X=\ell_1$\textup{)}.
	\item[\textup{(c)}] \textup{(}Choi-Song\textup{)} $X$ has property quasi-$\alpha$.
	\item[\textup{(d)}] \textup{(}Lindenstrauss\textup{)} $Y$ has property $\beta$ \textup{(}for instance, if $X$ is a closed subspace of $\ell_\infty$ containing the canonical copy of $c_0$\textup{)}.
	\item[\textup{(e)}] \textup{(}Acosta-Aguirre-Pay\'{a}\textup{)} $Y$ has property quasi-$\beta$.
	\item[\textup{(f)}] \textup{(}Iwanik\textup{)} $X=L_1(\mu)$ and $Y=L_1(\nu)$ where $\mu$ and $\nu$ are finite positive Borel measures.
	\item[\textup{(g)}] \textup{(}Finet-Pay\'{a}, Pay\'a-Saleh\textup{)} $X=L_1(\mu)$ and $Y=L_\infty(\nu)$ where $\mu$ is any measure and $\nu$ is a localizable measure.
	\item[\textup{(h)}] \textup{(}Johnson-Wolfe\textup{)} $X$ is Asplund and $Y=C(K)$ for a compact Hausdorff space $K$.
	\item[\textup{(i)}] \textup{(}Johnson-Wolfe\textup{)} In the real case, $X=C(K_1)$ and $Y=C(K_2)$ for compact Hausdorff spaces $K_1$ and $K_2$.
	\item[\textup{(j)}] \textup{(}Schachermayer\textup{)} $X=C(K)$ for a compact Hausdorff space $K$, and $Y$ does not contain $c_0$ isomorphically \textup{(}in particular, if $Y=L_p(\mu)$ with $1\leq p < \infty$ and $\mu$ is a positive measure\textup{)}.
	\item[\textup{(k)}] \textup{(}Uhl\textup{)} $X=L_1(\mu)$ for a $\sigma$-finite measure $\mu$, and $Y$ has the RNP.
	\item[\textup{(l)}] \textup{(}Cascales-Guirao-Kadets\textup{)} $X$ is Asplund and $Y$ is a uniform algebra \cite[Theorem 3.6]{CGK}.
	\item[\textup{(m)}] \textup{(}Acosta\textup{)} In the complex case, $X=C_0 (L)$ for a locally compact Hausdorff space $L$ and $Y$ is $\mathbb{C}$-uniformly convex \cite[Theorem 2.4]{A-C(K)}.
	\end{enumerate}
\end{examples}

Similarly, if all operators from a Banach space $X$ to a Banach space $Y$ are compact, then $\QNA(X,Y)=\Lin(X,Y)$ by Remark~\ref{basic_remark}, so this produce many examples, some of which we include in the following statement. For more examples of this kind, we refer the readers to \cite{A,AO,AO2,O} where one may find more related examples in the setting of Orlicz sequence spaces and Lorentz sequence spaces.

\begin{examples}\label{examples:K(X,Y)=L(X,Y)}\slshape
Any of the following conditions assures that $\K(X,Y)=\Lin(X,Y)$; hence $\QNA(X,Y)=\Lin(X,Y)$:
\begin{enumerate}
	\item[\textup{(a)}] \textup{(}Pitt\textup{)} $X$ is a closed subspace of $\ell_p$ and $Y$ is a closed subspace of $\ell_r$ with $1\leq r<p<\infty$.
	\item[\textup{(b)}] \textup{(}Rosenthal\textup{)} $X$ is a closed subspace of $c_0$ and $Y$ is a Banach space which does not contain $c_0$.
 	\item[\textup{(c)}] \textup{(}Rosenthal\textup{)} $X$ is a closed subspace of $L_p(\mu)$, $Y$ is a closed subspace of $L_r(\nu)$, $1\leq r<p<\infty$ and
   \begin{itemize}
   \item[\textup{(c1)}] $\mu$ and $\nu$ are atomic \textup{(}also covered by \textup{(}a\textup{)}\textup{)},
   \item[\textup{(c2)}] or $1\leq r<2$ and $\nu$ is atomic,
   \item[\textup{(c3)}] or $p>2$ and $\mu$ is atomic.
    \end{itemize}
 	\end{enumerate}
\end{examples}

Examples \ref{examples:K(X,Y)=L(X,Y)}.(a) can be found in \cite[Theorem~2.1.4]{Albiac-Kalton}, for instance; (b) and (c) can be found in a 1969 paper by H.~Rosenthal \cite{Rosenthal-JFA1969} as, respectively, Remark~4 and Theorem~A2.

With these examples in hands, we may present easily an example of a pair of Banach spaces $(X,Y)$ such that $\QNA(X,Y)$ is dense in $\Lin(X,Y)$ while $\NA(X,Y)$ is not.

\begin{example}\label{example:QNA-butNA-not-first}
{\slshape For $1<p<\infty$, $p\neq 2$, there exist closed subspaces $X$ of $c_0$ and $Y$ of $\ell_p$ such that $\NA(X,Y)$ is not dense in $\Lin(X,Y)$, while $\Lin(X,Y)=\K(X,Y)$, so $\QNA(X,Y)=\Lin(X,Y)$.\ }
\end{example}

We will use the following lemma from \cite{M}.

\begin{lemma}[\mbox{\cite[Proposition~4]{Lindenstrauss}, \cite[Lemma~2]{M}}]\label{lemma:JFA2014}
Let $X$ be a closed subspace of $c_0$ and let $Y$ be a strictly convex Banach space. Then, every operator in $\NA(X,Y)$ has finite rank.
\end{lemma}

Let us comment that for $X=c_0$ this was shown in \cite[Proposition~4]{Lindenstrauss} and the general case appeared in \cite[Lemma~2]{M}, but we would like to include here the easy argument for completeness. Indeed, fix $T\in \NA(X,Y)$ such that $\|T\|=\|Tx_0\|$ for $x_0\in S_X$ and use that $x_0\in X\leq c_0$ to get $N\in \N$ satisfying that $x_0(n)\leq \frac{1}{2}$ for every $n\geq N$; then, if $z\in X$ belongs to $$Z=\bigcap\nolimits_{1\leq k\leq N}\{x\in X\colon x(k)=0\},$$ and $\|z\|\leq \frac{1}{2}$, we have $\|x_0 \pm z\|\leq 1$, so $\|T(x_0)\pm T(z)\|\leq 1$; as $Y$ is strictly convex, it follows that $T(z)=0$ for every $z\in Z$, so $T$ is of finite-rank.

\begin{proof}[Proof of Example~\ref{example:QNA-butNA-not-first}]
Take $Y$ failing the approximation property (use \cite[Theorem~2.d.6]{Linden-Tz}, \cite[Theorem~1.g.4]{LTII}, for instance) and use the classical result by Grothendieck (see \cite[Theorem~18.3.2]{Jarchow} for instance) to get a closed subspace $X$ of $c_0$ such that $\K(X,Y)$ is not contained in the closure of the set of finite rank operators. Then, Lemma~\ref{lemma:JFA2014} gives that $\NA(X,Y)$ is not dense. Finally, that $\Lin(X,Y)=\K(X,Y)$ follows from Examples~\ref{examples:K(X,Y)=L(X,Y)}.
\end{proof}

To give examples of pairs $(X,Y)$ such that $\QNA(X,Y)$ is dense while not every operator from $X$ to $Y$ is compact is much easier: just consider the pairs $(c_0,c_0)$, $(\ell_1,\ell_1)$, and many others given by Examples~\ref{examples:NA-dense}. Nevertheless, to give an example of pair $(X,Y)$ such that $\QNA(X,Y)$ is dense in $\Lin(X,Y)$, $\NA(X,Y)$ is not dense, and $\Lin(X,Y)$ does not coincide with $\K(X,Y)$ is more involved. We will produce many examples of this kind in Section~\ref{section:RNP-sufficient}, see Examples~\ref{examples:QNAdense-NAnot}.

Let us now present some negative results on the denseness of quasi norm attaining operators. To do so we will need the following result which has been suggested to us by R.~Pay\'{a}. Recall that a \emph{monomorphism} between two Banach spaces $X$ and $Y$ is an operator $T\in \Lin(X,Y)$ which is an isomorphism from $X$ onto $T(X)$. It is well-known that $T\in \Lin(X,Y)$ is a monomorphism if and only if there is $C>0$ such that $\|Tx\|\geq C\|x\|$ for all $x\in X$, and if and only if $\ker T=\{0\}$ and $T(X)$ is closed (see \cite[\S~10.2.3]{Kadets_book}, for instance).

\begin{lemma}\label{lem_paya}
	Let $X$ and $Y$ be Banach spaces. If $T \in \QNA (X,Y)$ is a monomorphism, then $T \in \NA(X,Y)$.
\end{lemma}

\begin{proof}
Consider $T^{-1}\colon T(X)\longrightarrow X$. Let $(x_n) \subseteq S_X$ and $u \in \|T \| S_Y$ be such that $T x_n \longrightarrow u$. As $T(X)$ is closed, we have that $u\in T(X)$ so we may consider $x_0=T^{-1}u\in X$. Observe that $$x_n = T^{-1} (Tx_n) \longrightarrow T^{-1} u=x_0,$$ hence $\|x_0\|=1$. Besides, $Tx_0=T(T^{-1}u)=u$, so $\|Tx_0\|=\|u\|=\|T\|$ and $T\in \NA(X,Y)$, as desired.
\end{proof}

We will extend Lemma \ref{lem_paya} in Section~\ref{sect:further}, where we will also show that the hypothesis of monomorphism cannot be relaxed to the injectivity of the operator (see Example~\ref{example:Gowers-injective-QNA-no-NA}).

As an easy consequence of Lemma~\ref{lem_paya} we get the following result which will be the key to derive all of our negative results.

\begin{lemma}\label{lemma:paya-denseness}
Let $X$ and $Y$ be Banach spaces. If $T \in \Lin(X,Y)$ is a monomorphism such that $T \notin \overline{\NA (X,Y)}$, then $T \notin \overline{\QNA(X,Y)}$.
\end{lemma}

\begin{proof}
	Assume to the contrary that $T \in \overline{\QNA(X,Y)}$. Then we may choose a sequence $(T_n) \subseteq \QNA(X,Y)$ of monomorphisms which converges to $T$ since the set of all monomorphisms from $X$ to $Y$ is open in $\Lin(X,Y)$ (see \cite[Lemma~2.4]{Abramovich-Aliprantis}, for instance). By Lemma~\ref{lem_paya}, each $T_n \in \NA(X,Y)$; hence, $T \in \overline{\NA(X,Y)}$ which contradicts the given assumption on $T$.
\end{proof}

We are now able to get the negative results. As we commented in the introduction (see Example~\ref{example1:QNA_not_dense_G}), the first example can be deduced from the results in \cite{G}: $\QNA(c_0, Z)$ is not dense in $\Lin (c_0, Z)$ when $Z$ is an equivalent renorming of $c_0$ with the Kadec-Klee property. Next, we generalize the result of Lindenstrauss to the quasi norm attainment case where the statement is somewhat similar to the former one. Note that there are strictly convex equivalent renorming of $c_0$ which does not have the Kadec-Klee property (see \cite[Theorem 1 in p.~100]{DiesGeom}, for instance).

\begin{example}\label{example2:QNA_not_dense_G}
{\slshape	Let $X$ be an infinite-dimensional subspace of $c_0$ and let $Y$ be a strictly convex renorming of $c_0$. Then, $\QNA(X,Y)$ is not dense in $\Lin(X,Y)$. In particular, $\QNA(c_0,Y)$ is not dense in $\Lin(c_0,Y)$.}

Indeed, as $\NA(X,Y)$ is contained in the set of finite-rank operators by Lemma~\ref{lemma:JFA2014}, we get that the inclusion from $X$ to $Y$ (which is a monomorphism) does not belong to $\overline{\QNA(X,Y)}$ by Lemma~\ref{lemma:paya-denseness}.
\end{example}

The next example extends the result \cite[Corollary 2]{JJ2}, providing a new example such that $\QNA(X,Y)$ is not dense in $\Lin(X,Y)$.

\begin{example}\label{exam:negativeL1C(S)}
\slshape
There exists a compact Hausdorff space $S$ such that $\QNA(L_1[0,1],C(S))$ is not dense in $\Lin(L_1[0,1],C(S))$.
\end{example}

\begin{proof}
We basically follow the arguments in \cite{JJ2}. Let $S$ be given by the weak$^*$-closure in $L_\infty[0,1]$ of the set
	\[
	S_0 = \left\{ \sum_{i=1}^n \left( 1 - \frac{1}{2^i} \right) \chi_{D_i} \colon D_1, \ldots, D_n \subseteq [0,1] \text{ are disjoint}, \,\mu(D_i) < \frac{1}{2^i} \right\},
	\]
and define $T_0 \in \Lin(L_1[0,1],C(S))$ by
	\[
	\bigl[T_0 f\bigr](s) := \int_0^1 f(t) s(t) \,dt \qquad \text{for } s \in S, \, f \in L_1[0,1]
	\]
as in \cite[Corollary~2]{JJ2}. It is proved there that $T_0$ cannot be approximated by elements in $\NA(L_1[0,1],C(S))$, so the result follows from Lemma~\ref{lemma:paya-denseness} if we prove that $T_0$ is a monomorphism. Indeed, choose any $f \in L_1[0,1]$, and we may assume by taking $-f$ if necessary that
	\[
	\int_A f \,d\mu \geq \frac{1}{2} \|f\|_1
	\]
where $A = \{ w \in [0,1] \colon f(w) \geq 0 \}$. Let $(f_n) \subseteq L_1[0,1]$ be a sequence of nonzero simple functions such that
	\[
	\|f_n-f\|< \frac{1}{n} \qquad \text{for every } n \in \N.
	\]
For each $n \in \N$, we can choose $I_n \in \N$ and $D_1, \ldots, D_{I_n} \subset A$ with $\mu(D_i) < 1/{2^i}$ for $i = 1, \ldots, I_n$ such that
	\[
	\mu\left(A \setminus \bigcup\nolimits_{i=1}^{I_n} D_i \right) < \frac{1}{4}\frac{\|f\|_1}{\|f_n\|_\infty}.
	\]
Now, consider $s = \sum_{i=1}^I \left(1 - \frac{1}{2^i} \right) \chi_{D_i} \in S$. Then, we have that
\begin{align*}
\bigl| [T_0 f_n] (s) \bigr| & = \left| \int_0^1 f_n(w) s(w) \, d\mu(w) \right| \\
& = \int_{\bigcup_{i=1}^{I_n} D_i} f_n(w) s(w) \,d\mu(w) \geq \frac{1}{2} \int_{\bigcup_{i=1}^{I_n} D_i} f_n(w) \, d\mu(w) \\
& = \frac{1}{2} \left[ \int_A f_n(w) \,d\mu(w) - \int_{A \setminus \bigcup_{i=1}^{I_n} D_i} f_n(w) \,d\mu(w) \right] \\
& \geq \frac{1}{2} \left( \int_A f(w) \,d\mu(w) - \frac{1}{n} \right) - \frac{1}{2} \,\mu\left(A \setminus \bigcup\nolimits_{i=1}^{I_n} D_i\right) \, \|f_n\|_\infty \\
& \geq \frac{1}{4}\|f\|_1 - \frac{1}{2n} - \frac{1}{8}\|f\|_1 = \frac{1}{8} \|f\|_1 - \frac{1}{2n},
\end{align*}
which implies that $\|T_0 f_n\| \geq \frac{1}{8} \|f\|_1 - \frac{1}{2n}$ for each $n \in \N$. As $n$ tends to $\infty$, we obtain that $\|T_0f\| \geq \frac{1}{8} \|f\|_1$. It follows that $T_0$ is a monomorphism between $L_1[0,1]$ and $C(S)$ since $f \in L_1[0,1]$ was arbitrary.
\end{proof}

The last negative result that we want to present is related to the Radon-Nikod\'{y}m property. It was proved by R.~Huff \cite{H}, extending previous results of J.~Bourgain \cite{B2}, that if a Banach space $X$ fails to have the RNP, then there exist Banach spaces $X_1$ and $X_2$, which are both isomorphic to $X$, such that the formal identity from $X_1$ to $X_2$ cannot be approximated by elements of $\NA(X_1,X_2)$. Combining this fact with Lemma~\ref{lemma:paya-denseness}, we have just obtained the following result, stronger than Huff's one.
	
	\begin{prop}\label{prop:huff}
	If a Banach space $X$ does not have the RNP, then there exist Banach spaces $X_1$ and $X_2$ both isomorphic to $X$ such that $\QNA(X_1,X_2)$ is not dense in $\Lin(X_1,X_2)$.
	\end{prop}

It follows from Examples~\ref{examples:NA-dense}.(a) that the above result is actually a characterization of the RNP. We will show in the next section stronger characterizations of the RNP using quasi norm attaining operators which are not valid for norm attaining operators (see Corollary~\ref{corollary:RNP-equivalence}).

\section{The relation with the Radon-Nikod\'{y}m property: a new positive result}\label{section:RNP-sufficient}

We begin this section with our main result. Recall that a closed convex subset $D$ of a Banach space $X$ is said to be an \emph{RNP set} if every subset of $D$ is dentable. Observe that a Banach space $X$ has the RNP if and only if every closed bounded convex subset of $X$ is an RNP set (see \cite{Bourgin}). Given Banach spaces $X$ and $Y$, $T\in \Lin(X,Y)$ is a \emph{strong Radon-Nikod\'{y}m operator} (\emph{strong RNP operator} in short) if $\overline{T(B_X)}$ is an RNP set.

\begin{theorem}\label{theo:RNP-QNA}
Let $X$ and $Y$ be Banach spaces. Let $\eps>0$ be given and let $T \in \Lin(X,Y)$ be a strong RNP operator. Then, there exists $S \in \QNA(X,Y)$ such that
\begin{enumerate}
\item[\textup{(i)}] $\|S-T\|<\eps$,
\item[\textup{(ii)}] there exists $z_0 \in \overline{S(B_X)} \cap \|S\|S_Y$ such that whenever $(x_n)\subseteq B_X$ satisfies that $\|Sx_n\| \longrightarrow \|S\|$, we may find a sequence $(\theta_n) \subseteq \mathbb{T}$ such that $S(\theta_nx_n) \longrightarrow z_0$; in particular, there is $\theta_0 \in \mathbb{T}$ and a subsequence $(x_{\sigma(n)})$ of $(x_n)$ such that $Sx_{\sigma(n)} \longrightarrow \theta_0z_0$.
\end{enumerate}
\end{theorem}

In order to give a proof of Theroem~\ref{theo:RNP-QNA}, we need a deep result proved by C.~Stegall, usually known as the Bourgain-Stegall non-linear optimization principle. For a Banach space $Y$, a point $y_0$ of a bounded subset $D \subseteq Y$ is a \emph{strongly exposed point} if there is $y^*\in Y^*$ such that whenever a sequence $(y_n)\subseteq D$ satisfies that $\lim_n y^*(y_n)=\sup\{y^*(y)\colon y\in D\}$, $y_n$ converges to $y_0$ (in particular, $y^*(y_0)=\sup\{y^*(y)\colon y\in D\}$). In this case, we say that $y^*$ \emph{strongly exposes} $D$ at $y_0$ and that $y^*$ is a \emph{strongly exposing functional} for $D$ at $y_0$.

\begin{lemma}[\mbox{Bourgain-Stegall non-linear optimization principle, \cite[Theorem~14]{S2}}]\label{lemma:RNP-strongly-expose}
Suppose $D$ is a bounded RNP set of a Banach space $Y$ and $\phi \colon D \longrightarrow \R$ is upper semicontinuous and bounded above. Then, the set
\[
\{ y^* \in Y^* \colon \phi + \re y^* \textup{ strongly exposes } D \}
\]
is a dense $G_\delta$ subset of $Y^*$.
\end{lemma}

\begin{proof}[Proof of Theorem~\ref{theo:RNP-QNA}]
Assume that $\|T\| \neq 0$. As $\overline{T(B_X)}$ is an RNP set, by Lemma~\ref{lemma:RNP-strongly-expose} applied to the function $\phi(y)=\|y\|$ for every $y\in D=\overline{T(B_X)}$, there exists $y_0^* \in Y^*$ with $\|y_0^*\|<\eps/\|T\|$ such that $\|\cdot\| + \re y_0^*$ strongly exposes $\overline{T(B_X)}$ at some $y_0 \in \overline{T(B_X)}$. Then,
\[
\|y\| + \re y_0^*(y) \leq \|y_0\| + \re y_0^*(y_0) \qquad \text{for all } y \in \overline{T(B_X)}.
\]
By rotating $y \in \overline{T(B_X)}$, we also obtain that
\begin{equation}\label{eq:str-exp}
\|y\| + |y_0^*(y)| \leq \|y_0\| + \re y_0^*(y_0) \qquad \text{for all } y \in \overline{T(B_X)},
\end{equation}
and we have that $y_0^*(y_0)=|y_0^*(y_0)|$. Besides, if $(y_n) \subseteq \overline{T(B_X)}$ satisfies that
\[
\|y_n\| + \re y_0^*(y_n) \longrightarrow \|y_0\| + \re y_0^*(y_0),
\]
then $y_n \longrightarrow y_0$. So, if $(y_n) \subseteq \overline{T(B_X)}$ satisfies that
\begin{equation}\label{eq:str-exp2}
\|y_n\| + |y_0^*(y_n)| \longrightarrow \|y_0\| + y_0^*(y_0),
\end{equation}
then we can find $(\theta_n) \subseteq \mathbb{T}$ so that $\theta_ny_n \longrightarrow y_0$.

Now, define $S \in \Lin(X,Y)$ by
\[
Sx := Tx + y_0^*(Tx) \, \frac{y_0}{\|y_0\|} \qquad \text{for every } x \in X.
\]
It is easy to see that $\|S-T\| \leq \|y_0^*\|\|T\| < \eps$ and that
\[
\|Sx\| \leq \|Tx\| + |y_0^*(Tx)| \leq \|y_0\| + y_0^*(y_0) \qquad \text{for all } x \in B_X
\]
by \eqref{eq:str-exp}. Write $z_0=\left( 1+ \frac{y_0^*(y_0)}{\|y_0\|} \right) y_0$ and observe that $\|z_0\|=\|y_0\|+y_0^*(y_0)$. Now, take a sequence $(x_n) \subseteq B_X$ such that $Tx_n \longrightarrow y_0 \in \overline{T(B_X)}$, and observe that
\[
Sx_n = Tx_n + y_0^*(Tx_n)\, \frac{y_0}{\|y_0\|} \longrightarrow \left( 1+ \frac{y_0^*(y_0)}{\|y_0\|} \right) y_0=z_0.
\]
It follows that $\|S\| = \|y_0\| + y_0^*(y_0)=\|z_0\|$ and thus, $S \in \QNA(X,Y)$. Moreover, if $(z_n) \subseteq B_X$ satisfies $\|Sz_n\| \longrightarrow \|S\|$, then we have
\begin{align*}
\|y_0\| + y_0^*(y_0) & = \lim_n \left\| Tz_n + y_0^*(Tz_n)\, \frac{y_0}{\|y_0\|} \right\| \\
& \leq \lim_n \bigl( \|Tz_n\| + |y_0^*(Tz_n)| \bigr) \\
& \leq \|y_0\| + y_0^*(y_0) \quad \text{by }\eqref{eq:str-exp}.
\end{align*}
Thus by applying the equation \eqref{eq:str-exp2}, we can find $(\theta_n) \subseteq \mathbb{T}$ such that $T(\theta_n z_n) \longrightarrow y_0$, and hence
\[
S(\theta_n z_n)=T(\theta_n z_n) + y_0^*\bigl(T(\theta_n z_n)\bigr) \,\frac{y_0}{\|y_0\|} \longrightarrow \left( 1+ \frac{y_0^*(y_0)}{\|y_0\|} \right) y_0=z_0.\qedhere
\]
\end{proof}

Some remarks on the operator constructed in the proof of Theorem~\ref{theo:RNP-QNA} are pertinent. We first need the following easy result which relates the quasi norm attainment with the norm attainment of the adjoint operator. We will provide some comments and extensions of this result in Section~\ref{sect:further}. Recall that the \emph{adjoint operator} $T^* \colon Y^* \longrightarrow X^*$ of an operator $T \in \Lin(X,Y)$ is defined by $[T^* y^*] (x) := y^* (Tx)$. A simple calculation shows that $T^* \in \Lin(Y^*,X^*)$ with $\|T^*\|=\|T\|$.

\begin{prop}\label{prop:QNAimpliesNAadjoint}
	Let $X$ and $Y$ be Banach spaces and $T\in \Lin(X,Y)$. If $T \in \QNA(X,Y)$, then $T^* \in \NA(Y^*,X^*)$. Moreover, if $T$ quasi attains its norm toward $y_0\in \|T\|S_Y$, then $T^*$ attains its norm at any $y^* \in S_{Y^*}$ such that $y^*(y_0)=\|y_0\|$.
\end{prop}

\begin{proof}
	Let $(x_n) \subseteq B_X$ be such that $T x_n \longrightarrow y_0 \in \|T\| S_Y$. Take $y^* \in S_{Y^*}$ with $|y^* (y_0) | =\|y_0\|= \| T\|$. Then
	\[
	\|T^* y^* \| \geq \bigl|[T^* y^*] (x_n) \bigr| = | y^* (T x_n) | \longrightarrow |y^* (y_0)| = \|T\|,
	\]
	which implies that $\|T^* y^* \| = \|T\|$.
\end{proof}

We are now able to present the remarks on the construction given in the proof of Theorem~\ref{theo:RNP-QNA}.

\begin{rem}\label{remark:RNP-rankone-S*attains}
Let $X$ and $Y$ be Banach spaces and let $T \in \Lin(X,Y)$ be a strong RNP operator. Consider for $\eps>0$ the point $y_0\in \overline{T(B_X)}\subset Y$ and the operator $S \in \QNA(X,Y)$ given in the proof of Theorem~\ref{theo:RNP-QNA}. Then,
\begin{enumerate}
\item[\textup{(a)}] $T-S$ is a rank-one operator.
\item[\textup{(b)}] $S(B_X) \subseteq T(B_X) + \{\lambda y_0\colon \lambda\in \KK,\, |\lambda|\leq \rho\}$ for some $\rho>0$.
\item[\textup{(c)}] $\overline{S(X)}=\overline{T(X)}$.
\item[\textup{(d)}] $S$ quasi attains its norm toward a point of the form $z_0=\lambda y_0$ for some $\lambda>0$.
\item[\textup{(e)}] $S^*$ attains its norm at some $z^* \in S_{Y^*}$ which strongly exposes $\overline{S(B_X)}$ at $z_0$.
\end{enumerate}
\end{rem}

\begin{proof}
We only have to prove (e), the other assertions are obvious from the proof. Let $z_0 \in \overline{S(B_X)} \cap \|S\|S_Y$ be the point satisfying the condition stated in Theorem~\ref{theo:RNP-QNA}.(ii) and take any $z^*\in S_{Y^*}$ such that $z^*(z_0)=\|z_0\|=\|S\|$. As $S\in \QNA(X,Y)$ quasi attains its norm toward $z_0$, Proposition~\ref{prop:QNAimpliesNAadjoint} gives that $S^*$ attains its norm at $z^*$.

We claim that $z^*$ strongly exposes $\overline{S(B_X)}$ at $z_0$. Indeed, suppose that $(z_n) \subseteq \overline{S(B_X)}$ satisfies that $$\re z^* (z_n) \longrightarrow \sup\{\re z^*(y)\colon y\in \overline{S(B_X)}\} = \|z_0\|.$$ Choose $(x_n) \subseteq B_X$ such that $\|Sx_n - z_n\| < 1/n$ for every $n \in \N$. Observe that $z^*(Sx_n) \longrightarrow \|z_0\|$ and so, in particular, $\|Sx_n\| \longrightarrow \|S\|$. By Theorem~\ref{theo:RNP-QNA}.(ii), there is a sequence $(\theta_n) \subseteq \mathbb{T}$ such that $S(\theta_n x_n) \longrightarrow z_0$.
Since $(1-\theta_n) z^*(Sx_n) \longrightarrow 0$ and $z^*(Sx_n)\longrightarrow \|z_0\|\neq 0$, we obtain that $\theta_n \longrightarrow 1$. Therefore, we deduce that $Sx_n \longrightarrow z_0$. Hence, $z_n \longrightarrow z_0$ as desired.
\end{proof}

As a consequence of Theorem~\ref{theo:RNP-QNA} and Examples~\ref{examples:NA-dense}.(a), if either $X$ or $Y$ has the RNP, then $\QNA(X,Y)$ is dense in $\Lin(X,Y)$.

\begin{corollary}\label{coro:XRNPYRNP-QNAdense}
Let $X$, $Y$ be Banach spaces. If $X$ or $Y$ has the RNP, then $\QNA(X,Y)$ is dense in $\Lin(X,Y)$.
\end{corollary}

This cover the case when at least one of the spaces $X$ or $Y$ is reflexive. Actually, in this case the result is also covered by the next statement which follows from Theorem~\ref{theo:RNP-QNA}, the well-known fact that weakly compact convex sets are RNP sets (see \cite{Bourgin}, for instance), and Remark~\ref{remark:RNP-rankone-S*attains}.(a).

\begin{corollary}\label{coro:weaklycompact}
For every Banach spaces $X$ and $Y$, $\overline{\QNA(X,Y)\cap \W(X,Y)}=\W(X,Y)$.
\end{corollary}

We may present now examples of pairs of Banach spaces $(X,Y)$ for which $\QNA(X,Y)$ is dense in $\Lin(X,Y)$ while $\NA(X,Y)$ is not and that not every operator is compact.

\begin{examples}~\label{examples:QNAdense-NAnot}\slshape
\begin{enumerate}
 \item[\textup{(a)}] Let $Y$ be a strictly convex infinite-dimensional Banach space with the RNP (in particular, let $Y=\ell_p$ with $1<p<\infty$). Then, there is a Banach space $X$ such that $\NA(X,Y)$ is not dense in $\Lin(X,Y)$ and $\K(X,Y)$ does not cover the whole of $\Lin(X,Y)$, while $\QNA(X,Y)$ is.
 \item[\textup{(b)}] There is a Banach space $X$ such that $\NA(X,\ell_1)$ is not dense and $\K(X,\ell_1)$ does not cover the whole of $\Lin(X,\ell_1)$, while $\QNA(X,\ell_1)$ is dense in $\Lin(X,\ell_1)$.
 \item[\textup{(c)}] In the complex case, given a decreasing sequence $w\in c_0\setminus \ell_1$ of positive numbers,
let $d(w,1)$ be the corresponding Lorentz sequences space and let $d_*(w,1)$ be its natural predual. If $w\in \ell_2$, then $\NA(d_*(w,1),d(w,1))$ is not dense in $\Lin(d_*(w,1),d(w,1))$ while $\QNA(d_*(w,1),d(w,1))$ is dense as $d(w,1)$ has the RNP.
\end{enumerate}
\end{examples}

Examples \ref{examples:QNAdense-NAnot}.(a) follows from \cite[Theorem 2.3]{Aco-strictly-convex} (for $Y=\ell_p$ is actually consequence of \cite[Appendix]{G2}) and Examples \ref{examples:QNAdense-NAnot}.(b) follows from \cite[Theorem 2.3]{Aco-L1}. For (c), we refer to \cite[\S 4]{ChoiKaminskaLee} for the definitions and basic properties of the spaces; the result on non-denseness of $\NA(d_*(w,1),d(w,1))$ appears in \cite[\S 4]{ChoiKaminskaLee}, while the RNP of $d(w,1)$ is immediate as it is a separable dual space.

We do not know whether there is a Banach space $Z$ such that $\QNA(Z,Z)$ is dense in $\Lin(Z,Z)$ while $\NA(Z,Z)$ is not dense, see Problem~\ref{problem:ellinftysum}.

We are now ready to present another important consequence of Theorem~\ref{theo:RNP-QNA}: two characterization of the RNP in terms of denseness of quasi norm attaining operators. Note that the RNP is an isomorphic property and so it is a sufficient condition to get the universal denseness conditions on equivalent renormings.

	\begin{corollary}\label{corollary:RNP-equivalence}
	Let $Z$ be a Banach space. Then, the following are equivalent.
	\begin{enumerate}
	\item[\textup{(a)}] $Z$ has the RNP.
	\item[\textup{(b)}] For every Banach space $Y$ and every equivalent renorming $Z'$ of $Z$, $\QNA(Z',Y)$ is dense in $\Lin(Z',Y)$.
	\item[\textup{(c)}] For every Banach space $X$ and every equivalent renorming $Z'$ of $Z$, $\QNA(X,Z')$ is dense in $\Lin(X,Z')$.
	\end{enumerate}
	\end{corollary}

	\begin{proof}
(a)$\Rightarrow$(b) is a consequence of \cite[Theorem 5]{B2} (here is Examples~\ref{examples:NA-dense}.(a)). (b)$\Rightarrow$(a) and (c)$\Rightarrow$(a) can be obtained from Proposition~\ref{prop:huff}. Finally, (a)$\Rightarrow$(c) follows from Theorem~\ref{theo:RNP-QNA}.
	\end{proof}

	Let us comment that in the case of norm attaining operators in the classical sense, while the analogous statement of (a) and (b) are equivalent, the same is not true for statement (c), as (a)$\Rightarrow$(c) does not hold (see Examples~\ref{examples:QNAdense-NAnot}). Therefore, the use of quasi norm attaintment gives a symmetry in the characterization of the Radon-Nikod\'{y}m property which is not possible for the classical norm attainment.

We next would like to take the advantage of Theorem~\ref{theo:RNP-QNA} in particular cases. To this end, we introduce some terminologies here.

\begin{definition}\label{definition:strong-sense}
Let $X$ and $Y$ be Banach spaces. We say that $T\in \Lin(X,Y)$ \emph{uniquely quasi attains its norm} if there is $u \in Y$ such that every sequence $(x_n)\subset B_X$ satisfying $\|Tx_n\|\longrightarrow \|T\|$ has a subsequence $(x_{\sigma(n)})$ satisfying that $Tx_{\sigma(n)}\longrightarrow \theta u$ for some $\theta \in \T$. In this case, we will say that $T$ uniquely quasi attains its norm \emph{toward} $u$ and that $T$ is a \emph{uniquely quasi norm attaining} operator.
\end{definition}

While it is obvious that norm attaining operators are quasi norm attaining (Remark~\ref{basic_remark}), it is not true that norm attaining operators are uniquely quasi norm attaining: the identity on any Banach space of dimension greater than one is clearly an example. Also, the fact that $B_X$ is an RNP set does not imply that so is $\overline{T(B_X)}$: indeed, consider a surjective map $T \in \Lin (\ell_1, c_0)$ with $\overline{T(B_{\ell_1})} = B_{c_0}$ (see \cite[Theorem~5.1]{FHHMZ} for example). Therefore, a version of Corollary~\ref{coro:XRNPYRNP-QNAdense} for uniquely quasi norm attainment does not follow from Theorem~\ref{theo:RNP-QNA} for the RNP in the domain space. However, the following notion introduced by J.~Bourgain in \cite[p.~268]{B2}, extending the concept of strongly exposing functional, forces uniquely quasi norm attainment. For Banach spaces $X$ and $Y$, an operator $T\in \Lin(X,Y)$ \emph{absolutely strongly exposes} the set $B_X$ ($T$ is an \emph{absolutely strongly exposing} operator) if there exists $x \in B_{X}$ such that whenever a sequence $(x_n) \subset B_X$ satisfies that $\lim_{n} \|Tx_n\| = \|T\|$, there is a subsequence $(x_{\sigma(n)})$ which converges to $\theta x$ for some $\theta\in \T$. It is clear that an absolutely strongly exposing operator $T \in \Lin (X,Y)$ is uniquely quasi norm attaining. It follows then from \cite[Theorem 5]{B2} that if $B_X$ is an RNP set (i.e.\ $X$ has the RNP), then the set of uniquely quasi norm attaining operators is dense in $\Lin (X,Y)$. Therefore, the following result follows from this fact and Theorem~\ref{theo:RNP-QNA}.

\begin{corollary}\label{coro:XRNPYRNP-QNAdense-strongversion}
Let $X$ and $Y$ be Banach spaces. If either $X$ or $Y$ has the RNP, then the set of uniquely quasi norm attaining operators from $X$ to $Y$ is dense in $\Lin(X,Y)$.
\end{corollary}

Two more consequences of Theorem~\ref{theo:RNP-QNA} can be stated for compact operators and weakly compact operators (as compact and weakly compact sets are RNP sets) with the help of Remark~\ref{remark:RNP-rankone-S*attains}.(a), as it follows that when we start with a compact or weakly compact operator $T$ in the proof of Theorem~\ref{theo:RNP-QNA}, the operator $S$ that it is obtained is compact or weakly compact, respectively. In the case of compact operators, we already know that they quasi attain their norms (see Remark~\ref{basic_remark}.(b)), but the result is still interesting as it is stronger than that (observe that the identity on a two-dimensional Banach space is compact but it is not uniquely quasi norm attaining).

\begin{corollary}\label{coro:weaklycompact_compact_absolutelystronglyexposingQNA}
Let $X$ and $Y$ be Banach spaces.
\begin{enumerate}
\item[\textup{(a)}] Compact operators from $X$ to $Y$ which uniquely quasi attain their norm are dense in $\K (X,Y)$.
\item[\textup{(b)}] Weakly compact operators from $X$ to $Y$ which uniquely quasi attain their norm are dense in $\W (X,Y)$.
\end{enumerate}
\end{corollary}

The following result is an immediate consequence of Corollary~\ref{coro:XRNPYRNP-QNAdense-strongversion}, which will be useful in many applications later on.

\begin{corollary}\label{corollary:Gamma-RNP-dense}
Let $X$ and $Y$ be Banach spaces, $\Gamma \subseteq B_X$ such that $\overline{\co}(\T\Gamma)=B_X$ and suppose that $Y$ has the RNP. Then, for every $T\in \Lin(X,Y)$ and every $\eps>0$, there is $S\in \Lin(X,Y)$ with $\|T-S\|<\eps$ and a sequence $(x_n)\subseteq \Gamma$ such that $(Sx_n)$ converges to some $y_0\in \|S\|S_Y$. In other words, the set
\[
\bigl\{T \in \Lin(X,Y) \colon Tx_n \longrightarrow y_0 \textup{ for some } (x_n) \subseteq \Gamma \textup{ and } y_0 \in \|T\|S_Y \bigr\}
\]
is dense in $\Lin(X,Y)$.
\end{corollary}

\begin{proof}
Given $T \in \Lin (X,Y)$, since $Y$ has the RNP, Corollary~\ref{coro:XRNPYRNP-QNAdense-strongversion} provides an operator $S\in \QNA(X,Y)$ such that $\|T-S\|<\eps$ and $z_0\in \|S\|S_Y$ such that given any sequence $(x_n)\subseteq B_X$ with $\|Sx_n\|\longrightarrow \|S\|$, then there exists a subsequence $(x_{\sigma(n)})$ and $\theta_0\in \T$ such that $Sx_{\sigma(n)}\longrightarrow \theta_0 z_0$. Now, as $$\sup\{\|Sx\|\colon x\in \Gamma\}=\|S\|$$ by the assumption, we may find a sequence $(x_n)\subseteq \Gamma$ such that $\|Sx_n\|\longrightarrow \|S\|$ and the result follows.
\end{proof}

The rest of this section is devoted to the applications of Corollary \ref{corollary:Gamma-RNP-dense} in some situations: norm attainment on Lipschitz maps, multilinear maps, and homogeneous polynomials.

\subsection{An application: Lipschitz maps attaining the norm toward vectors}

Recall from Definition~\ref{def:LipA} that given real Banach spaces $X$ and $Y$, $\LipA(X,Y)$ denotes the set of those Lipschitz maps $f\in \Lip(X,Y)$ for which there exist a sequence of pairs $\bigl((x_n,y_n)\bigr) \subseteq \widetilde{X}$ and $u\in \|f\|_\text{Lip}S_Y$ satisfying
	\[
	\frac{f(x_n)-f(y_n)}{\|x_n-y_n\|} \longrightarrow u.
	\]

As a consequence of Corollary~\ref{coro:XRNPYRNP-QNAdense-strongversion}, we obtain the following result.

\begin{corollary}\label{corollary:RNP-LDA-dense}
Let $X$ and $Y$ be real Banach spaces such that $Y$ has the RNP. Then, $\LipA(X,Y)$ is dense in $\Lip(X,Y)$.
\end{corollary}

We need to introduce some terminology. Recall that the \emph{Lipschitz-free space} $\F(X)$ over $X$ is a closed linear subspace of $\Lip(X,\R)^*$ defined by
	\[
	\F(X) := \overline{\spann}\{ \delta_x \colon x \in X\},
	\]
where $\delta_x$ is the canonical point evaluation of a Lipschitz map $f$ at $x$ given by $\delta_x(f) := f(x)$. We refer the reader to the paper \cite{Godefroy-survey-2015} and the book \cite{weaver2} (where it is called Arens-Eells space) for a detailed account on Lipschitz free spaces. The following properties of $\F(X)$ can be found there. It is well known that $\F(X)$ is indeed an isometric predual of $\Lip(X,\R)$. Moreover, for any Lipschitz map $f \in \Lip(X,Y)$, we can define a unique bounded linear operator $T_f\in \Lin(\F(X),Y)$ by $T_f(\delta_x) := f(x)$, which satisfies that $\|T_f\|=\|f\|_{\text{Lip}}$; furthermore, $\Lip(X,Y)$ is isometrically isomorphic to $\Lin(\F(X),Y)$ via this correspondence between $f$ and $T_f$. We define a set of \emph{molecules} of $X$ by
	\[
	\operatorname{Mol}(X) := \left\{ m_{x,y} := \frac{\delta_x-\delta_y}{\|x-y\|} \colon (x,y) \in \widetilde{X} \right\} \subseteq \F(X).
	\]
An easy consequence of the Hahn-Banach theorem is that $B_{\F(X)} = \overline{\co}(\operatorname{Mol}(X))$.

\begin{proof}[Proof of Corollary~\ref{corollary:RNP-LDA-dense}]
Let $\eps >0$ and $f \in \Lip (X,Y)$ be given. As $Y$ has the RNP, applying Corollary~\ref{corollary:Gamma-RNP-dense} for the set $\Gamma=\operatorname{Mol}(X)$, there exists $G \in \QNA(\F(X), Y)$, a sequence $\left(m_{p_n, q_n}\right) \in \operatorname{Mol}(X)$, and $y_0\in \|G\|S_Y$ such that
$$
\|T_f - G\| < \eps \quad \text{and} \quad G\left(m_{p_n, q_n}\right) \longrightarrow y_0.
$$
If we take (the unique) $g \in \Lip(X,Y)$ such that $T_g = G$, then $\|f-g\|_{\text{Lip}}=\|T_f-G\|<\eps$ and
	\[
	\frac{g(p_{n})-g(q_{n})}{\|p_{n}-q_{n}\|}=G\left(m_{p_n, q_n}\right) \longrightarrow y_0.
	\]
This shows that $g \in \LipA(X,Y)$ and it completes the proof.
\end{proof}

\subsection{An application: quasi norm attaining multilinear maps and polynomials}
Let $X_1, \ldots, X_n$ and $Y$ be Banach spaces. The set of all bounded $n$-linear maps from $X_1 \times \cdots \times X_n$ to $Y$ will be denoted by $\Lin (X_1, \ldots,X_n;Y)$. As usual, we define the norm of $A \in \Lin (X_1, \ldots, X_n ; Y)$ by
\[
\|A\| = \sup \bigl\{ \|A(x_1, \ldots, x_n) \| \colon (x_1, \ldots, x_n) \in B_{X_1} \times \cdots \times B_{X_n} \bigr\}.
\]
The following definition is a natural extension of quasi norm attainment from linear operators to multilinear maps.

\begin{definition}\label{definition:QNA-multilinear}
We say that $A \in \Lin(X_1,\ldots,X_n;Y)$ \emph{quasi attains its norm} (in short, $A \in \QNA(X_1,\ldots,X_n;Y)$) if $$\overline{A(B_{X_1}\times \cdots \times B_{X_n})}\cap \|A\|S_Y \neq \emptyset,$$
or, equivalently, if there exist a sequence $\bigl(x^{(1)}_m, \ldots, x^{(n)}_m\bigr)\subseteq B_{X_1} \times \cdots \times B_{X_n}$ and a point $u \in \|A\|S_Y$ such that
	\[
A\bigl(x^{(1)}_m,\ldots,x^{(n)}_m\bigr) \longrightarrow u.
	\]
In this case, we say that $A$ \emph{quasi attains its norm toward} $u$.
\end{definition}

Here, the consequence of Corollary~\ref{corollary:Gamma-RNP-dense} in this setting is the following.

\begin{corollary}\label{corollary:multilinearMaps}
Let $X_1,\ldots X_n$ and $Y$ be Banach spaces. If $Y$ has the RNP, then $\QNA(X_1,\ldots,X_n;Y)$ is dense in $\Lin(X_1,\ldots,X_n;Y)$.
\end{corollary}

Analogously to what happens with Lipschitz maps, we present a way to linearize multilinear maps: the projective tensor product. Given Banach spaces $X_1,X_2,\ldots,X_n$, the \emph{projective tensor product} of the spaces, which will be denoted by $X_1 \otimes_\pi \cdots \otimes_\pi X_n$, is the space $X_1 \otimes \cdots \otimes X_n$ endowed with the \emph{projective norm} $\pi$. We write $X_1 \widetilde{\otimes}_\pi\cdots \widetilde{\otimes}_\pi X_n$ for its completion. It is well known that given any $A\in \Lin(X_1,\ldots,X_n;Y)$, there is a unique $\widehat{A}\in \Lin\left(X_1 \widetilde{\otimes}_\pi \cdots \widetilde{\otimes}_\pi X_n,Y\right)$ such that
\[
\widehat{A}(x_1\otimes \cdots \otimes x_n)=A(x_1,\ldots,x_n) \qquad \text{for all } x_1\otimes \cdots \otimes x_n\in {X_1}\widetilde{\otimes}_\pi \cdots \widetilde{\otimes}_\pi {X_n}.
\]
Moreover, the spaces $\Lin(X_1,\ldots,X_n;Y)$ and $\Lin\left(X_1 \widetilde{\otimes}_\pi \cdots \widetilde{\otimes}_\pi X_n,Y\right)$ are isometrically isomorphic through this correspondence and the closed unit ball of $X_1 \widetilde{\otimes}_\pi \cdots \widetilde{\otimes}_\pi X_n$ is the absolutely closed convex hull of $B_{X_1}\otimes \cdots \otimes B_{X_n} := \{ x_1 \otimes \cdots \otimes x_n : x_i \in B_{X_i}, 1\leq i \leq n \}$ (see \cite[Proposition~16.8]{DeGaMaSe}, for instance).

\begin{proof}[Proof of Corollary~\ref{corollary:multilinearMaps}]
Let $\eps>0$ and $A \in \Lin(X_1,\ldots,X_n;Y)$ be given. Let $\widehat{A} \in \Lin(X_1 \widetilde{\otimes}_\pi \cdots \widetilde{\otimes}_\pi X_n,Y)$ be the corresponding linear operator on the tensor product space defined as above. Now, it follows by Corollary \ref{corollary:Gamma-RNP-dense} applied to $\Gamma := B_{X_1} \otimes \cdots \otimes B_{X_n} \subseteq B_{X_1 \widetilde{\otimes}_\pi \cdots \widetilde{\otimes}_\pi X_n}$, that there exist $\widehat{S} \in \QNA(X_1 \widetilde{\otimes}_\pi\cdots \widetilde{\otimes}_\pi X_n,Y)$, a sequence $\bigl(x_1^{(m)} \otimes \cdots \otimes x_n^{(m)}\bigr)_{m\in\N} \subseteq \Gamma$ and a point $u \in \|\widehat{S}\|S_Y$ such that $\|\widehat{S}-\widehat{A}\|<\eps$ and $\widehat{S}\bigl(x_1^{(m)} \otimes \cdots \otimes x_n^{(m)}\bigr) \longrightarrow u$. Define a $n$-linear map $S \colon X_1 \times \cdots \times X_n \longrightarrow Y$ by $S(x_1, \ldots, x_n) := \widehat{S}(x_1 \otimes \cdots \otimes x_n)$. Then, we can conclude that $S(x_1^{(m)}, \ldots, x_n^{(m)}) \longrightarrow u$ with $\|u\|=\|\widehat{S}\|=\|S\|$ and that $\|S-A\|<\eps$.
\end{proof}

Analogously, we can get a similar result on quasi norm attaining polynomials. Let $X$ and $Y$ be Banach spaces, and let $n \in \N$ be given. Recall that an $n$-linear mapping $A \in \Lin ({X, \overset{n}{\ldots}, X}; Y)$ is said to be \emph{symmetric} if $A(x_1, \dots, x_n) = A(x_{\sigma(1)}, \dots, x_{\sigma(n)})$ for any $(x_1, \dots, x_n) \in X \times \overset{n}{\cdots} \times X$ and any permutation $\sigma$ of the set $\{1, \dots, n\}$. We let $\Lin_s ({X, \overset{n}{\ldots}, X}; Y)$ denote the space of all bounded symmetric $n$-linear mappings endowed with the norm
\[
\|A\| = \sup \bigl\{ \|A(x, \ldots, x) \| \colon x \in B_X \bigr\}.
\]
A mapping $P \colon X \longrightarrow Y$ is an \emph{$n$-homogeneous polynomial} if there is a symmetric $n$-linear map $\widecheck{P} \in \Lin_s({X, \overset{n}{\ldots}, X};Y)$ such that $P(x)=\widecheck{P}(x, \ldots, x)$. We denote by $\mathcal{P}(^nX;Y)$ the space of all continuous $n$-homogeneous polynomials from $X$ to $Y$ equipped with the usual norm $\| P \| = \sup_{x \in B_X} \| Px\|$ for $P \in \mathcal{P}(^n X; Y)$.

\begin{definition}\label{definition:QNA-polynomial}
We say that $P \in \mathcal{P}(^nX;Y)$ \emph{quasi attains its norm} (in short, $P \in \QNA(\mathcal{P}(^nX;Y)))$ if
	\[
	\overline{P(B_X)} \cap \|T\| S_Y \neq \emptyset,
	\]
or equivalently, if there exist a sequence $(x_m) \subseteq B_X$ and a point $u \in \|P\|S_Y$ such that $Px_m \longrightarrow u$.
\end{definition}

Here, the consequence of Corollary~\ref{corollary:Gamma-RNP-dense} in this setting is the following.

\begin{corollary}\label{corollary:polynomials}
Let $X$ and $Y$ be Banach spaces, and let $n \in \N$ be given. If $Y$ has the RNP, then $\QNA(\mathcal{P}(^nX;Y))$ is dense in $\mathcal{P}(^nX;Y)$.
\end{corollary}

The proof is almost the same as that for Corollary \ref{corollary:multilinearMaps}.
We denote by $X \otimes_{\pi_s} \overset{n}{\cdots} \otimes_{\pi_s} X$ the \emph{$n$-fold symmetric tensor product} of a Banach space $X$, endowed with the \emph{symmetric tensor norm} $\pi_s$, and by $X \widetilde{\otimes}_{\pi_s} \overset{n}{\cdots} \, \widetilde{\otimes}_{\pi_s} X$ its completion. We refer the reader to the Chapter 16 of the recent book \cite{DeGaMaSe} for a detailed account on symmetric tensor product spaces. Note that the spaces $\mathcal{P}(^nX;Y)$ and $\Lin (X \widetilde{\otimes}_{\pi_s} \overset{n}{\cdots} \, \widetilde{\otimes}_{\pi_s} X, Y)$ are isometrically isomorphic by means of the correspondence $P \in \mathcal{P}(^nX;Y) \longmapsto \widetilde{P} \in \Lin (X \widetilde{\otimes}_{\pi_s} \overset{n}{\cdots} \, \widetilde{\otimes}_{\pi_s} X, Y)$, where
\[
\widetilde{P} (x \otimes \cdots \otimes x) = P(x) \qquad \text{for all } x \in X,
\]
see \cite[Proposition 16.23]{DeGaMaSe} for instance. Moreover, the closed unit ball of $X \widetilde{\otimes}_{\pi_s} \overset{n}{\cdots} \, \widetilde{\otimes}_{\pi_s} X$ is the absolutely closed convex hull of $B_{X}\otimes \overset{n}{\cdots} \otimes B_{X} := \{ x \otimes \cdots \otimes x\colon x \in B_{X}\}$, see again \cite[Proposition 16.23]{DeGaMaSe}.

\begin{proof}
Let $\eps >0$ and $P \in \mathcal{P}(^nX;Y)$ be given. Consider the linear operator $\widetilde{P} \in \Lin (X \widetilde{\otimes}_{\pi_s} \overset{n}{\cdots} \, \widetilde{\otimes}_{\pi_s} X, Y)$ corresponding to $P$ as explained above. Applying Corollary \ref{corollary:Gamma-RNP-dense} to $\Gamma := B_X \otimes \overset{n}{\cdots} \otimes B_{X}$, we may find $\widetilde{Q} \in \QNA (X \widetilde{\otimes}_{\pi_s} \overset{n}{\cdots} \, \widetilde{\otimes}_{\pi_s} X, Y)$, a sequence $(x_m \otimes \cdots \otimes x_m) \subset \Gamma$ and a point $u \in \| \widetilde{Q} \| S_Y$ such that $\| \widetilde{Q} - \widetilde{P} \| < \eps$ and $\widetilde{Q} (x_m \otimes \cdots \otimes x_m) \longrightarrow u$. Considering the polynomial $Q \in \mathcal{P}(^nX;Y)$ induced by $\widetilde{Q}$, we have that $Qx_m \longrightarrow u$ with $\|u \| = \| \widetilde{Q} \| = \| Q \|$ and that $\|Q - P\| < \eps$, finishing the proof.
\end{proof}

\section{On the set of quasi norm attaining operators}\label{section:reflexivity-finite-dimension}

Given Banach spaces $X$ and $Y$, our goal here is to analyze when any of the inclusion
$$
\NA(X,Y)\subseteq \QNA(X,Y) \subseteq \Lin(X,Y)
$$
can be an equality. For the first inclusion, the equality allows to characterize reflexivity in terms of quasi norm attaining operators.

\begin{prop}\label{prop:reflexive}
	Let $X$ be a Banach space. Then the following statements are equivalent:
	\begin{enumerate}
	\item[\textup{(a)}] $X$ is reflexive.
	\item[\textup{(b)}] $\NA(X,Y) = \QNA (X,Y)$ for every Banach space $Y$.
	\item[\textup{(c)}] There exists a nontrivial Banach space $Y$ such that $\NA(X,Y) = \QNA (X,Y)$.
\end{enumerate}
\end{prop}

\begin{proof} (a)$\Rightarrow$(b). Suppose that $X$ is reflexive and
	let $T \in \QNA (X,Y)$ be given. Let $(x_n) \subseteq S_X$ be a sequence such that $Tx_n \longrightarrow u$ for some $u \in Y$ with $\|u\| = \|T\|$. Since $X$ is reflexive, there is a subsequence, denoted again by $(x_n)$, such that $x_n \longrightarrow x_0$ weakly for some $x_0 \in B_X$. Thus $T x_n \longrightarrow T x_0$ weakly and $Tx_0$ must coincide with $u$, getting that $\|Tx_0\|=\|u\|=\|T\|$ and so $T\in \NA(X,Y)$.
	
	Since (b)$\Rightarrow$(c) is clear, it remains to show that (c)$\Rightarrow$(a). Fix any $x^*\in X^*$ and $y_0\in S_Y$, and consider the rank-one operator $T=x^*\otimes y_0$ given by $T(x):=x^*(x)y_0$. It is clear that $\|T\|=\|x^*\|$ and that $T\in \NA(X,Y)$ if and only if $x^*\in \NA(X,\KK)$. On the other hand, as rank-one operators are compact, $T\in \QNA(X,Y)$ by Remark~\ref{basic_remark}.(b). Now, if $\NA(X,Y) = \QNA (X,Y)$, it follows that $\NA(X,\KK)=X^*$, and this implies that $X$ is reflexive by James' theorem (see \cite[Corollary~3.131]{FHHMZ} for instance).
\end{proof}

The version for range spaces of the above result does not give any interesting characterization.

\begin{remark}
{\slshape There is no nontrivial Banach space $Y$ such that $\NA(X,Y)=\QNA(X,Y)$ for every Banach space $X$.}\newline Indeed, just consider a non-reflexive Banach space $X$ and suppose that $\NA(X,Y)=\QNA(X,Y)$. Then, Proposition~\ref{prop:reflexive} gives the contradiction.
\end{remark}

Our next aim is to discuss the equality in the inclusion $\QNA(X,Y)=\Lin(X,Y)$ to show that finite-dimensionality can also be described in terms of the set of quasi norm attaining operators: every infinite-dimensional Banach space can be the domain or the range of bounded linear operators which does not quasi attain their norms. The case of the domain space is easier and it is based on \cite[Lemma~2.2]{M-M-P}.

\begin{prop}\label{prop_X_c0}
	If $X$ is an infinite dimensional Banach space, then there is $T\in \Lin (X,c_0)$ which does not belong to $\QNA(X, c_0)$.
\end{prop}

\begin{proof}
	By the Josefson-Nissenzweig theorem (see \cite[\S~XII]{Diestel_seq_series}), there exists a weak$^*$ null sequence $(x_n^*) \subseteq S_{X^*}$. Define an operator $T \colon X \longrightarrow c_0$ by
	\[
T(x)=\left(\frac{m}{m+1} x_m^* (x)\right)_{m\in\N} \qquad \text{for } x \in X.
	\]
	It is clear that $\|T \| =1$. Assume that there exist a sequence $(x_n) \subseteq S_X$ and a vector $u \in S_{c_0}$ such that $T x_n \longrightarrow u = (u_m)_{m\in\N}$. Then, for each $m \in \N$,
	\[
	\frac{m}{m+1} x_m^* (x_n) \longrightarrow u_m
	\]
	as $n \longrightarrow \infty$. Therefore,
	\[
	|u_m| = \lim_{n \rightarrow \infty} \left|\frac{m}{m+1} x_m^* (x_n)\right| \leq \frac{m}{m+1} < 1
	\]
	for every $m \in \N$. If we take $m_0 \in \N$ such that $|u_{m_0}| = \| u \| =1$ (which exists as $u\in c_0$), we get a contradiction.
\end{proof}

The case of the range space is analogous, but the proof is rather more involved.

\begin{prop}\label{prop_ell1_Y}
Let $Y$ be an infinite dimensional Banach space. Then there exists $T\in \Lin (\ell_1,Y)$ which does not belong to $\QNA(\ell_1, Y)$.
\end{prop}

\begin{proof}
	We divide the proof into two cases.

	\emph{Case 1}: Assume first that $Y$ doesn't have the Schur property. Then, it is well-known that there is a sequence $(y_n) \subseteq S_Y$ such that $y_n \xrightarrow[]{\ w\ } 0$ (see \cite[Exercise~XII.2]{Diestel_seq_series}, for instance). Choose an increasing sequence $(t_n)$ with $0<t_n<1$ for each $n \in \N$ such that $\lim_n t_n=1$. Consider an operator $T \colon \ell_1 \longrightarrow Y$ defined as
	\[
	T(x) = \sum_{n=1}^\infty t_n x(n) y_n \qquad \text{for each } x\colon \N\longrightarrow \KK \text{ in }\ell_1.
	\]
	It is clear that $\|T \| = 1$. Assume $T$ belongs to $\QNA(\ell_1, Y)$, then there exist a sequence $(x_{m}) \subseteq S_{\ell_1}$ and a vector $u \in S_Y$ such that $\lim_m Tx_{m} = u$. Choose $u^* \in S_{Y^*}$ so that $|u^* (u) | =1$. We have that
	\begin{equation}\label{eq_ell1_Y}
	u^* (T x_{m}) = \sum_{n=1}^\infty t_n x_m(n) u^* (y_n) \longrightarrow u^* (u) \qquad \text{as } m \longrightarrow \infty.
	\end{equation}
	As $y_n \xrightarrow[]{\ w\ } 0$, we can choose $n_0 \in \N$ so large that $|u^* (y_n)| < 1/2$ for all $n > n_0$ and $t_{n_0} \geq 1/2$. It follows that
	\[
	\left| \sum_{n=1}^\infty t_n x_m(n) u^* (y_n) \right| \leq t_{n_0} \sum_{n=1}^{n_0} |x_m(n)| + \frac{1}{2} \sum_{n= n_0 + 1}^{\infty} |x_m(n)| \leq t_{n_0} < 1
	\]
	for every $m \in \N$, which contradicts \eqref{eq_ell1_Y}. Let us comment that the above argument is inspired in the proof of \cite[Remark~3]{MartinRao}.
	
	\emph{Case 2}: Now suppose that $Y$ has the Schur property. It follows from Rosenthal's $\ell_1$ theorem that $Y$ contains a subspace which is isomorphic to $\ell_1$ (see \cite[Exercise~XI.3]{Diestel_seq_series} for instance), that is, there exists a monomorphism $Q \colon \ell_1 \longrightarrow Y$, so $\ker Q=\{0\}$ and $Q(\ell_1)$ is closed. Write $u_n = Q(e_n)$ for every $n \in \N$, then there is $C>0$ such that $C<\| u_n \| \leq \|Q\|$ for every $n \in \N$. Take an increasing sequence $(t_n)$ with $1/2 < t_n < 1$ for every $n \in \N$ such that $\lim_n t_n=1$.
	Consider the operator $T \colon \ell_1 \longrightarrow Q(\ell_1) \subseteq Y$ defined by
	\[
	T(x) = \sum_{n=1}^{\infty} t_n x(n) \frac{u_n}{\|u_n\|} \qquad \text{for each } x\colon \N\longrightarrow \KK \text{ in }\ell_1.
	\]
	It is clear that $\|T \| =1$, $\ker T=\ker Q=\{0\}$, and $T(\ell_1)=Q(\ell_1)$. Therefore, $T$ is a monomorphism.
	If $T \in \QNA(\ell_1, Y)$, then Lemma~\ref{lem_paya} implies that $T \in \NA(\ell_1, Y)$. However, if $\|Tx\|=1$ for some $x \in S_{\ell_1}$, then
	\[
	1 = \|Tx\| \leq \sum_{n=1}^\infty |t_n||x(n)| < \sum_{n=1}^\infty |x(n)| =1,
	\]
	which is a contradiction. Thus, we conclude that $T \notin \QNA(\ell_1, Y)$ as desired.	
\end{proof}

From Propositions \ref{prop_X_c0} and \ref{prop_ell1_Y}, we obtain the following characterization of finite-dimensionality in terms of quasi norm attaining operators.

\begin{corollary}\label{corollary:finite-dimension}
let $Z$ be a Banach space. Then, the following assertions are equivalent:
\begin{enumerate}
	\item[\textup{(a)}] $Z$ is finite-dimensional.
	\item[\textup{(b)}] $\QNA(Z,Y) = \Lin(Z,Y)$ for every Banach space $Y$.
	\item[\textup{(c)}] $\QNA(Z,c_0) = \Lin(Z,c_0)$.
	\item[\textup{(d)}] $\QNA(X,Z) = \Lin(X,Z)$ for every Banach space $X$.
	\item[\textup{(e)}] $\QNA(\ell_1,Z) = \Lin(\ell_1,Z)$.
\end{enumerate}
\end{corollary}

\begin{proof}
The fact that (a) implies the rest of assertions follows from Remark~\ref{basic_remark}; (b) $\Rightarrow$ (c) and (d) $\Rightarrow$ (e) are immediate. Finally, (c) $\Rightarrow$ (a) and (e) $\Rightarrow$ (a) are consequences of Propositions~\ref{prop_X_c0} and \ref{prop_ell1_Y}, respectively.
\end{proof}

Some remarks on the previous result are pertinent. Observe that if $T\in \Lin(X,Y)\setminus \QNA(X,Y)$ has norm one, then the set $K=\overline{T(B_X)}$ is contained in the open unit ball of $Y$ but $\sup_{y \in K} \|y\|=1$. This phenomena has a relation with the so-called remotality. A bounded subset $E$ of a Banach space $X$ is said to be \emph{remotal from $x\in X$} if there is $e_x\in E$ such that $\|x-e_x\|=\sup\{\|x-e\|\colon e\in E\}$, and $E$ is said to be \emph{remotal} if it is remotal from all elements in $X$ (see \cite{Baronti-Papini} for background). Notice that up to translating and re-scaling, the existence of a non-remotal subset of a Banach space is equivalent to the existence of a subset $E$ of the open unit ball such that $\sup\{\|e\|\colon e\in E\}=1$. The existence of non-remotal (closed) sets in every infinite-dimensional Banach space is easy to prove (see \cite[Remark~3.2]{Baronti-Papini}). But it seems that the existence of closed \emph{convex} non-remotal subsets in every infinite-dimensional Banach space was an open problem until 2009-2010, when it was proved independently in two different papers \cite[Theorem~7]{MartinRao} and \cite[Proposition~2.1]{Vesely}. Furthermore, in 2011, a new and easier proof was given in \cite{Kraus2011}. Observe that Proposition~\ref{prop_ell1_Y} gives a new proof of this fact: indeed, if $Y$ is infinite-dimensional, there exists $T\in \Lin(\ell_1,Y)\setminus \QNA(\ell_1,Y)$, and we may suppose
that $\|T\|=1$, so $K=\overline{T(B_{\ell_1})}$ is contained in the open unit ball of $Y$ and $\sup\{\|y\|\colon y\in K\}=1$. Moreover, the non-remotal set $K$ given by this result is not only closed and convex but closed and absolutely convex (i.e.\ convex and equilibrated). Having a look to the previous proofs, it is not difficult to adapt the ones in \cite{MartinRao} and \cite{Vesely} to get an absolutely convex set in the real case, while the one in \cite{Kraus2011} does not give absolute convexity even in the real case. For the complex case, the proof of \cite{Vesely} is not adaptable at all, while the one of \cite{MartinRao} seems to be. But the proof of Proposition~\ref{prop_ell1_Y} is simpler than the one of \cite{MartinRao}. In any case, as a by-product of our study of quasi norm attaining operators, we get the following corollary.

\begin{corollary}\label{coro:remotality}
Let $Y$ be a \textup{(}real or complex\textup{)} infinite-dimensional Banach space. Then there exists a closed, absolutely convex subset of $B_Y$ which is not remotal from $0$.
\end{corollary}

\section{Further results and examples}\label{sect:further}
Our aim in this section is to provide a more extensive study of two results previously stated: first, the relationship between the quasi norm attainment and the norm attainment of the adjoint operator (discussing extensions of Proposition~\ref{prop:QNAimpliesNAadjoint}) and, second, the possible extensions of Lemma~\ref{lem_paya} on conditions assuring that quasi norm attainment implies (classical) norm attainment.

We begin with the relationship among the adjoint operators. Our first result is a characterization of quasi norm attaining weakly compact operators in terms of the adjoint and biadjoint. In particular, it shows that Proposition~\ref{prop:QNAimpliesNAadjoint} is a characterization in this case. Before stating the result, notice from Remark~\ref{basic_remark}.(b) that the inclusion $\K(X,Y) \subseteq \QNA(X,Y)$ holds for all Banach spaces $X$ and $Y$ while this does not remain true when we replace $\K(X,Y)$ by $\W(X,Y)$; just recall that for every reflexive space $X$, and Proposition~\ref{prop_X_c0} provides an example of $T\in \W(X,c_0)=\Lin(X,c_0)$ which does not belong to $\QNA(X,c_0)$.

\begin{prop}\label{prop_W_QNA} Let $X$ and $Y$ be Banach spaces, and $T \in \W(X,Y)$. Then the following are equivalent.
\begin{enumerate}
\item[\textup{(a)}] $T \in \QNA(X,Y)$.
\item[\textup{(b)}] $T^{*}\in \NA(Y^*,X^{*})$.
\item[\textup{(c)}] $T^{**}\in \NA(X^{**},Y)$.
\end{enumerate}
\end{prop}

\begin{proof}
(a)$\Rightarrow$(b) follows from Proposition~\ref{prop:QNAimpliesNAadjoint} and (b)$\Rightarrow$(c) is immediate, so it suffices to prove (c)$\Rightarrow$(a). Pick $x_0^{**}\in S_{X^{**}}$ such that $\|T^{**}(x_0^{**})\|=\|T\|$ and consider a net $(x_\lambda)\subseteq B_X$ such that $J_X(x_\lambda) \xrightarrow[]{\ w^*\ } x_0^{**}$, where $J_X\colon X\longrightarrow X^{**}$ denotes the natural isometric inclusion map. Thus, we have that $$T x_\lambda=T^{**}(J_X(x_\lambda)) \xrightarrow[]{\ w^*\ } T^{**}(x_0^{**}).$$ But as $T$ is weakly compact, $T^{**}(x_0^{**})$ belongs to $Y$ so, actually, we have that $$T x_\lambda=T^{**}(J_X(x_\lambda)) \xrightarrow[]{\ w\ } T^{**}(x_0^{**})\in Y$$ and then, $T^{**}(x_0^{**}) \in \overline{T(B_X)}$. Therefore, $\overline{T(B_X)}\cap \|T\|S_Y\neq \emptyset$, that is, $T\in \QNA(X,Y)$.
\end{proof}

This proposition gives an alternative (and probably simpler) proof of the fact presented in Corollary~\ref{coro:weaklycompact} that weakly compact operators can be always approximated by weakly compact quasi norm attaining operators.

\begin{cor}\label{cor_W_QNA}
	Let $X$ and $Y$ be Banach spaces. Then $\overline{\QNA(X,Y)\cap \W(X,Y)}=\W(X,Y)$.
\end{cor}

\begin{proof}
	By observing the proof of \cite[Theorem~1]{Lindenstrauss}, we see that every $S\in \W(X,Y)$ can be approximated by a sequence $(S_n)$ of weakly compact operators such that $S_n^{**}\in \NA(X^{**},Y)$. By Proposition~\ref{prop_W_QNA}, each $S_n$ belongs to $\QNA (X,Y)$; hence $S$ belongs to $\overline{\QNA(X,Y)\cap \W(X,Y)}$.
\end{proof}

A sight to Proposition~\ref{prop_W_QNA} may lead us to think that Proposition~\ref{prop:QNAimpliesNAadjoint} is an equivalence, that is, that for all Banach spaces $X$, $Y$, one has that $T\in \QNA(X,Y)$ if (and only if) $T^*\in \NA(Y^*,X^*)$. However, this is not true in general. Indeed, it is known that the set $\{ T \in \Lin(X, Y) \colon T^* \in \NA (Y^*, X^*) \}$ is dense in $\Lin (X,Y)$ for all Banach spaces $X$ and $Y$ \cite{Z}; on the other hand, there exist (many) Banach spaces $X$ and $Y$ for which $\QNA(X,Y)$ is not dense in $\Lin(X,Y)$, see Examples
\ref{example2:QNA_not_dense_G}, \ref{exam:negativeL1C(S)}, \ref{examples:QNAdense-NAnot} or Proposition~\ref{prop:huff}. Our next result is an explicit example of this phenomenon.

\begin{example}\label{eample:adjoint-NA-operator-not-QNA}
{\slshape
Recall the Day's norm $\vertiii{\cdot}$ on $c_0$ (see \cite[Definition II.7.2]{DGZ}) defined as
\[
\vertiii{x} = \sup_n \left\{ \left( \sum_{k=1}^{n} \frac{|x_{\gamma_k}|^2}{4^k} \right)^{\frac{1}{2}} \right\},
\]
where $(x_{\gamma_n})$ is a decreasing rearrangement of $(x_n)$ with respect to the modulus. If we let $Y=(c_0, \vertiii{\cdot})$, it follows from Lemma~\ref{lemma:JFA2014} that the formal identity map $\Id \in \Lin (c_0, Y)$ does not belong to $\QNA (c_0, Y)$ since $\vertiii{\cdot}$ is strictly convex. But $\Id^* \in \Lin(Y^*,\ell_1)$ attains its norm at $z^* = \left( \frac{1}{\sqrt{3}} \cdot \frac{1}{2^n} \right) \in Y^*$.}

Indeed, from the construction of $Y$, one can derive with a few calculations that $\|\Id\| = \frac{1}{\sqrt{3}}$. Observe first that
	\begin{align*}
	\vertiii{z^*}_{Y^*} \geq \bigl| z^* (\underbrace{\sqrt{3}, \ldots, \sqrt{3}}_{N\,\, \text{many terms}}, 0, 0, \ldots ) \bigr| \geq \sum_{j=1}^{N} \frac{1}{2^j} \qquad \text{for each } N \in \N
	\end{align*}
	as $(\underbrace{\sqrt{3}, \ldots, \sqrt{3}}_{N\,\, \text{many terms}}, 0, 0, \ldots ) \in B_Y$. This implies that $\vertiii{z^*} \geq 1$.
	
	On the other hand,	we prove that $\vertiii{ z^*}_{\spann\{e_1,\ldots, e_N\}} \leq 1$ for every $N \in \N$ by an induction argument. It is obvious when $N =1$. For fixed $N \geq 2$, suppose that $\vertiii{z^* }_{\spann\{e_1,\ldots, e_n\}} \leq 1$ for every $1 \leq n \leq N-1$.
	Take $y = (y_1, \ldots, y_N, 0,0, \ldots) \in \spann\{e_1,\ldots,e_N\}$ with $\vertiii{y} \leq 1$ (so, $\|y\|_{\infty} \leq 2$) and denote by $\hat{y} = (\hat{y}_1, \ldots, \hat{y}_N, 0,0,\ldots)$ the decreasing rearrangement of $y$.
It follows that
	\begin{equation}\label{estimation:z*(y)}
	| z^* ( y ) | = \frac{1}{\sqrt{3}} \left| \sum_{j=1}^{N} \frac{y_j}{2^j} \right| \leq \frac{1}{\sqrt{3}} \sum_{j=1}^{N} \frac{\hat{y}_j}{2^j} = \frac{1}{\sqrt{3}} \left( \frac{\hat{y}_1}{2} + \sum_{j=2}^{N} \frac{\hat{y}_j}{2^j} \right).
	\end{equation}

	If $\hat{y}_1 \leq \sqrt{3}$, then we have from \eqref{estimation:z*(y)} that $|z^*(y)| \leq 1$. Suppose that $\sqrt{3} < \hat{y}_1 \leq 2$. Note that $(\hat{y}_2, \ldots, \hat{y}_N, 0,0,\ldots ) \in \spann\{e_1,\ldots, e_{N-1}\}$ and that
	\begin{align*}
	\vertiii{ (\hat{y}_2, \ldots, \hat{y}_N, 0,0,\ldots ) }^2 = \sum_{j=1}^{N-1} \frac{\hat{y}_{j+1}^2}{4^j} = 4 \sum_{j=2}^{N} \frac{\hat{y}_j^2}{4^j}.
	\end{align*}
	Thus,
	\begin{align*}
	\vertiii{y}^2 = \frac{\hat{y}_1^2}{4} + \sum_{j=2}^{N} \frac{\hat{y}_j^2}{4^j} = \frac{\hat{y}_1^2}{4} + \frac{1}{4} \vertiii{ (\hat{y}_2, \ldots, \hat{y}_N, 0,0,\ldots ) }^2 \leq 1
	\end{align*}
	which implies that $\vertiii{ (\hat{y}_2, \ldots, \hat{y}_N, 0,0,\ldots ) } \leq \sqrt{4-\hat{y}_1^2}$. By induction hypothesis, we have that
	\[
	 \frac{1}{\sqrt{3}} \sum_{j=1}^{N-1} \frac{\hat{y}_{j+1}}{2^j} = \frac{1}{\sqrt{3}} \sum_{j=2}^{N} \frac{\hat{y}_{j}}{2^{j-1}} \leq \sqrt{4-\hat{y}_1^2}.
	\]
	Combining this with \eqref{estimation:z*(y)}, we obtain that
	\begin{equation}\label{estimation2:z*(y)}
	|z^* (y)| \leq \frac{\hat{y}_1}{2\sqrt{3}} + \sqrt{1-\frac{\hat{y}_1^2}{4}}.
	\end{equation}
	Consider the function $g(t) = \frac{t}{2\sqrt{3}} + \sqrt{1-\frac{t^2}{4}}$ for $t \in [0,2]$. As $g(t) \leq 1$ for $\sqrt{3} \leq t \leq 2$, we obtain from \eqref{estimation2:z*(y)} that $|z^* (y) | \leq 1$. This finishes the induction process showing that $\vertiii{ z^*}_{\spann\{e_1,\ldots, e_N\}} \leq 1$ for every $N \in \N$, and we can deduce that $\vertiii{z^*} =1$. Finally, $\| \Id^* (z^*) \| = \frac{1}{\sqrt{3}} = \| \Id^* \|$ gives the desired result.
\end{example}

We now want to characterize quasi norm attaining operators in terms of the norm attainment of the adjoint operator.

\begin{prop}\label{prop:equivalenceQNA-adjointNA}
\slshape Let $X$, $Y$ be Banach spaces and let $T\in \Lin(X,Y)$. Then, the following are equivalent:
\begin{itemize}
 \item[\textup{(a)}] $T\in \QNA(X,Y)$,
 \item[\textup{(b)}] $T^*$ attains its norm at some $y^*\in S_{Y^*}$ for which $|y^*|$ attains its supremum on $\overline{T(B_X)}$.
\end{itemize}
\end{prop}

\begin{proof}
(a)$\Rightarrow$(b). Consider $(x_n)$ in $B_X$ such that $Tx_n\longrightarrow y_0\in \|T\|S_Y$ and take $y^*\in S_{Y^*}$ such that $y^*(y_0)=\|T\|$. Then
$$
\|T^*y^*\|\geq \bigl|[T^*(y^*)](x_n)\bigr|=|y^*(Tx_n)|\longrightarrow |y^*(y_0)|=\|T\|,
$$
so $\|T^*y^*\|=\|T^*\|$. Besides, the supremum of $|y^*|$ on $\overline{T(B_X)}$ is $\|T\|$ and it is attained at $y_0$.

(b)$\Rightarrow$(a). By hypothesis, there are $y^*\in S_{Y^*}$ and $y_0\in \overline{T(B_X)}$ such that $|y^*(y_0)|=\sup\left\{|y^*(y)|\colon y\in \overline{T(B_X)}\right\}$. Observe that as $\|T^*(y^*)\|=\|T\|$ it follows that the supremum of $|y^*|$ on $\overline{T(B_X)}$ equals $\|T\|$, so we get that $\|y_0\|=\|T\|$ and then $\overline{T(B_X)}\cap \|T\|S_Y\neq \emptyset$, that is, $T\in \QNA(X,Y)$.
\end{proof}

Observe that the condition in Proposition \ref{prop:equivalenceQNA-adjointNA}.(b) is weaker than the one given in Remark~\ref{remark:RNP-rankone-S*attains}.(e).

To finish the study of the norm attainment of the adjoint operator, we include the next straightforward result which shows that the quasi norm attainment and the norm attainment are equivalent for adjoint operators.

\begin{prop}\label{prop:T*-QNA=NA}
Let $X$ and $Y$ be Banach spaces and $T\in \Lin(X,Y)$. If $T^*\in \QNA(Y^*,X^*)$, then $T^*\in \NA(Y^*,X^*)$.
\end{prop}

\begin{proof}
Let $(y_n^*) \subseteq B_{Y^*}$ be a sequence such that $T^*y_n^*\longrightarrow x_0^*$ for some $x_0^*\in \|T^*\|S_{X^*}$. As $B_{Y^*}$ is weak$^{*}$ compact, there is a weak$^{*}$ accumulation point $y_0^*$ of $A=\{y_n^*\colon n\in \N\}$. Then $y_0^*\in B_{Y^*}$ and the only possibility is that $T^*y_0^*=x_0^*$ by the $w^{*}$-$w^{*}$-continuity of $T^*$. Hence, $\|T^*y_0^*\|=\|T^*\|$.
\end{proof}

Next, we would like to deal with possible extensions of Lemma~\ref{lem_paya}, looking for sufficient conditions assuring that quasi norm attaining operators are actually norm attaining. We start with the following useful characterization of quasi norm attaining operators.

\begin{lemma}\label{lem_quotient}
Let $X$ and $Y$ be Banach spaces and $T \in \Lin(X,Y)$ be given. Consider the following commutative diagram:
	\[
	\begin{tikzcd}
	X \arrow{r}{T} \arrow[swap]{d}{q} & Y \\
	X/\ker T \arrow[swap]{ur}{\widetilde{T}}
	\end{tikzcd}
	\]
	where $q \colon X \longrightarrow X/\ker T$ is the canonical quotient map and $\widetilde{T} \colon X/\ker T \longrightarrow Y$ is the induced (injective) operator. Then, the following are equivalent.
	\begin{enumerate}
	\item[\textup{(a)}] $T \in \QNA(X,Y)$.
	\item[\textup{(b)}] $\widetilde{T} \in \QNA(X/\ker T, Y)$.
	\end{enumerate}
	If $T(X)$ is closed, then the following statement is also equivalent to the others:
	\begin{enumerate}
	\item[\textup{(c)}] $\widetilde{T} \in \NA(X/\ker T, Y)$.
	\end{enumerate}
\end{lemma}

\begin{proof}
(a)$\Rightarrow$(b). Suppose that $T \in \QNA(X,Y)$. Let $(x_n) \subseteq S_X$ and $u \in \|T\| S_Y$ be such that $T x_n \longrightarrow u$. Note that $(q(x_n)) \subseteq B_{X/\ker T}$ and $\widetilde{T} (q(x_n)) \longrightarrow u$. As $\|\widetilde{T}\|=\|T\|$, this implies that $\widetilde{T} \in \QNA (X/\ker T, Y)$. Conversely, suppose that $\widetilde{T} \in \QNA (X/\ker T, Y)$. Let $(\tilde{x}_n) \subseteq S_{X/\ker T}$ and $u \in \|\widetilde{T}\| S_Y$ be such that $\widetilde{T} \tilde{x}_n \longrightarrow u$. As $q(\text{Int} (B_X) ) = \text{Int} (B_{X/\ker T})$, there exists a sequence $(x_n) \subseteq \text{Int} (B_X)$ so that $q(x_n) = \frac{n}{n+1} \tilde{x}_n$ for each $n \in \N$. Observe that
\[
T x_n = \widetilde{T} (q(x_n)) = \left(\tfrac{n}{n+1}\right) \widetilde{T} \tilde{x}_n \longrightarrow u,
\]
hence $T \in \QNA(X,Y)$.

In order to prove the last equivalence, just observe that if $T(X)$ is closed, then $\widetilde{T}$ is a monomorphism, so Lemma~\ref{lem_paya} gives the equivalence between $\widetilde{T} \in \QNA(X/\ker T, T(X))$ and $\widetilde{T} \in \NA(X/\ker T, T(X))$.
\end{proof}

Our next aim is to present conditions allowing us to get that the norm attainment and the quasi norm attainment are equivalent. In the first result, we start with an operator $T\in \Lin(X,Y)$ with closed range such that $T\in \QNA(X,Y)$ and get that $T\in \NA(X,Y)$ from the fact that $\widetilde{T}\in \NA(X/ \ker T,Y)$ using proximinality of the kernel of the operator. Recall that a (closed) subspace $Y$ of a Banach space $X$ is said to be \emph{proximinal} if for every $x \in X$ the set
\[
P_Y (x) := \{ y \in Y \colon \|x-y\| = \dist (x, Y) \}
\]
is nonempty (we refer to \cite{Singer} for background). An easy observation is that a hyperplane is proximinal if and only if it is the kernel of a norm attaining operator (see \cite[Theorem~2.1]{Singer}). It is well known (and easy to prove) that $Y$ is proximinal if and only if $q(B_X)=B_{X/Y}$ where $q$ is the quotient map from $X$ onto $X/Y$ \cite[Theorem~2.2]{Singer}. Another basic result on proximinality is that reflexive subspaces are proximinal in every superspace \cite[Corollary~2.1]{Singer}.

We are now able to present our first result, which is an extension of Lemma~\ref{lem_paya}.

\begin{prop}\label{prop_proximinal}
	Let $X$ and $Y$ be Banach spaces. If $T \in \QNA(X,Y)$ satisfies that $T(X)$ is closed and $\ker T$ is proximinal, then $T \in \NA(X,Y)$.
\end{prop}

\begin{proof}
	Since $T(X)$ is closed, $\widetilde{T} \in \NA(X/\ker T, Y)$ by Lemma~\ref{lem_quotient}. Let $\tilde{x} \in B_{X/\ker T}$ be a point so that $\| \widetilde{T} \tilde{x} \| = \| \widetilde{T}\|$. Now, by proximinality, there exists $x \in B_X$ such that $q(x)=\tilde{x}$. As $Tx=\widetilde{T}(q(x))=\widetilde{T}(\tilde{x})$, we get that $\| T x \| = \| \widetilde{T} \tilde{x} \| = \| \widetilde{T} \| = \|T\|$.
\end{proof}

We present some consequences of the above result. The first one follows from the fact that reflexive subspaces are proximinal in every superspace \cite[Corollary~2.1]{Singer}.

\begin{corollary}\label{corollary:kerTreflexive}
Let $X$, $Y$ be Banach spaces and let $T\in \Lin(X,Y)$. If $T\in \QNA(X,Y)$, $T(X)$ is closed and $\ker T$ is reflexive, then $T\in \NA(X,Y)$.
\end{corollary}

We next give some examples showing that the conditions of Proposition~\ref{prop_proximinal} are all necessary. The first example shows that proximinality cannot be dropped.

\begin{example}
{\slshape Let $X$ be a non-reflexive Banach space. Then every $f\in \Lin(X,\KK)$ belongs to $\QNA(X,\KK)$ by Remark~\textup{\ref{basic_remark}.(b)} and, clearly, $f(X)$ is closed. Nevertheless, there are elements in $\Lin(X,\KK)\setminus \NA(X,\KK)$ by James' theorem (see \cite[Corollary~3.131]{FHHMZ} for instance).}
\end{example}

The second example, more interesting, shows that being injective is not enough for a quasi norm attaining operator to be norm attaining. It also shows that closedness of the range of the operator is needed in both Proposition~\ref{prop_proximinal} and Corollary~\ref{corollary:kerTreflexive}. We need to present the so-called Gowers' space $G$ introduced in \cite[proof of Theorem in Appendix]{G2} (see also \cite[Example 7]{AcoAguPay-Rocky} or \cite{AcoAguPay-Carolinae} for our notation, some properties, and the obvious extension to the complex case of Gowers' results). For a sequence $x$ of scalars and $n\in \N$, we write
$$
\Phi_n(x)=\frac{1}{H_n} \sup \left\{\sum\nolimits_{j\in J} |x(j)|\colon J\subset \N,\, |J|=n \right\}
$$
where $|J|$ is the cardinality of the set $J$ and $H_n=\sum\nolimits_{k=1}^n k^{-1}$. \emph{Gowers' space} $G$ is the Banach space of those sequences $x$ satisfying that $$\lim_{n\to \infty} \Phi_n(x)=0$$ equipped with the norm given by
$$
\|x\|=\sup\bigl\{\Phi_n(x)\colon n\in \N\bigr\} \qquad \text{for } x\in G.
$$

\begin{example}\label{example:Gowers-injective-QNA-no-NA}
{\slshape Let $G$ be Gowers' space and given $1<p<\infty$, let $T \colon G \longrightarrow \ell_p$ be the formal identity map. Then, $T\in \QNA(G,\ell_p)$, $\ker T=\{0\}$, but $T \notin \overline{\NA (G,\ell_p)} \cup \K(G,\ell_p)$.}
\end{example}

\begin{proof}
From the proof of \cite[Theorem in Appendix]{G2}, we can see that $\| T \| = \left(\sum\nolimits_{i=1}^{\infty} i^{-p} \right)^{1/p}$ and that $T \notin \overline{\NA (G,\ell_p)} \cup \K(G,\ell_p)$. Consider $x_n = \left(1,\tfrac{1}{2}, \ldots, \tfrac{1}{n}, 0,0,\ldots\right) \in S_G$ for $n \in \N$. From the facts that $\|x_n\|_G =1$, $\| Tx_n \|_p = \left(\sum_{i=1}^{n} i^{-p} \right)^{1/p} $ for every $n \in \N$, and that $(Tx_n)$ converges to $\left(1,\frac{1}{2}, \ldots, \frac{1}{n}, \frac{1}{n+1},\ldots\right) \in \|T\|S_{\ell_p}$, we deduce that $T \in \QNA(G, \ell_p)$.
\end{proof}

In the next proposition, we give a connection between the set $\QNA(X,Y)$ and the lineability of the set $\NA(X,\KK)$. It is an extension of \cite[Proposition~2.5]{KLMW} where the result was proved for compact operators and, actually, its proof is based on the proof of that result.

\begin{prop}\label{prop:QNA=NA_lineability_NA}
	Let $X$ and $Y$ be Banach spaces and let $T\in \Lin(X,Y)$. If $T\in \QNA(X,Y)$ and $(\ker T)^\bot \subseteq \NA(X,\KK)$, then $T \in \NA(X,Y)$.
\end{prop}

\begin{proof}
Consider $\widetilde{T}\colon X/\ker T \longrightarrow Y$ as in Lemma~\ref{lem_quotient} and use this result to get that $\widetilde{T}\in \QNA(X/\ker T, Y)$. Then, by Proposition~\ref{prop:QNAimpliesNAadjoint}, we get that $\widetilde{T}^*\in \NA\bigl(Y^*,(X/\ker T)^*\bigr)$, so there is $y^*\in S_{Y^*}$ such that
$$
\|\widetilde{T}^*(y^*)\|=\|\widetilde{T}\|=\|T\|.
$$
Now, the functional $x^*=T^*(y^*)=\bigl[q^*\widetilde{T}^*\bigr](y^*)\in X^*$ vanishes on $\ker T$, so it belongs to $(\ker T)^\bot\subset \NA(X,\KK)$. This implies that there is $x\in S_X$ such that
$$
|x^*(x)|=\|x^*\|=\bigl\|\bigl[q^*\widetilde{T}^*\bigr](y^*)\bigr\| =\bigl\|q^*(\widetilde{T}^*y^*)\bigr\| = \bigl\|\widetilde{T}^*(y^*)\bigr\|=\|T\|,
$$
where we have used the immediate fact that $q^*$ is an isometric embedding as $q$ is a quotient map. Therefore, $\|T\|=|[T^*y^*](x)|=|y^*(Tx)|$ and so $\|Tx\|=\|T\|$, as desired.
\end{proof}

Observe that for a reflexive space $X$, Proposition~\ref{prop:QNA=NA_lineability_NA} reproves the result in Proposition~\ref{prop:reflexive} that $\QNA(X,Y)=\NA(X,Y)$ for every Banach space $Y$.

\section{Stabilities on quasi norm attaining operators}\label{section:stability}
The aim of this section is to present some results which allow to transfer the denseness of the set of quasi norm attaining operators from some pairs to other pairs. The first result is that the denseness is preserved by some kinds of absolute summands of the domain space and to every kind of absolute summands of the range space. Recall that an \emph{absolute sum} of Banach spaces $X$ and $Y$ is the product space $X \times Y$ endowed with the product norm $\|(x,y)\|_a:=|(\|x\|,\|y\|)|_a$ of an absolute norm, where $|\cdot|_a$ is an \emph{absolute norm} (i.e.\ a norm in $\R^2$ satisfying that $|(1,0)|_a=|(0,1)|_a=1$, $|(x,y)|_a=|(|x|,|y|)|_a$ for all $x,y\in \R$). A closed subspace $X_1$ of a Banach space $X$ is said to be an \emph{absolute summand} if $X=X_1\oplus_a X_2$ for some absolute sum $\oplus_a$ and some closed subspace $X_2$ of $X$. An absolute norm $|\cdot|_a$ is said to be of \emph{type} $1$ if $(1,0)$ is a vertex of $B_{(\R^2,|\cdot|_a)}$, that is, if the set $\{ (1,t) \in \R^2 \colon |(1,t)|_a=1 \}$ separates the points of $(\R^2,|\cdot|_a)$. We refer to \cite{CDJM} for the use of absolute sums related to norm attaining and to the reference given there for general background on absolute sums. Classical examples of absolute sums are the $\ell_p$-sums for $1\leq p\leq \infty$. In the case of $p=1$, $\oplus_1$ summands are usually known as $L$-summands and it is clear that they are of type $1$.

\begin{prop}\label{prop:main_stability_result}
Let $X$ and $Y$ be Banach spaces such that $\QNA(X,Y)$ is dense in $\Lin(X,Y)$.
\begin{enumerate}
\item[\textup{(a)}]
If $X_1$ is an absolute summand of $X$ of type $1$, then $\QNA(X_1,Y)$ is dense in $\Lin(X_1,Y)$.
\item[\textup{(b)}]
In particular, if $X_1$ is an $L$-summand of $X$, then $\QNA(X_1,Y)$ is dense in $\Lin(X_1,Y)$.
\item[\textup{(c)}]
If $Y_1$ is an absolute summand of $Y$, then $\QNA(X,Y_1)$ is dense in $\Lin(X,Y_1)$.
\end{enumerate}
\end{prop}

The proofs of these results are adaptation of the corresponding ones given in \cite{CDJM} for norm attaining operators.

\begin{proof}
(a) Let $T \in \Lin (X_1, Y)$ and $\eps >0$ be given. Write $X = X_1 \oplus_{a} X_2$ and define $\widetilde{T} \in \Lin (X, Y)$ as $\widetilde{T} (x_1, x_2) := T x_1$ for every $(x_1, x_2) \in X$. Then there exists $\widetilde{S} \in \QNA (X, Y)$ such that $\|\widetilde{S} \| = \| \widetilde{T} \|$ and $\| \widetilde{S} - \widetilde{T} \| < \eps$. Choose a sequence $(x_n) = (x_n^{(1)}, x_n^{(2)}) \subseteq S_{X}$ satisfying $\widetilde{S} x_n \longrightarrow u$ for some $u \in Y$ with $\| u \| = \|\widetilde{S}\|$.
Define $S \in \Lin (X_1, Y)$ as $S(x_1) := \widetilde{S} (x_1, 0)$ for every $x_1 \in X_1$, then $\|S \| \leq \| \widetilde{S} \|$ and
\[
\| Sx_1 - T x_1 \| = \| \widetilde{S} (x_1, 0) - \widetilde{T} (x_1, 0) \| \leq \| \widetilde{S} - \widetilde{T} \| < \eps
\]
for all $x_1 \in B_{X_1}$. Thus $\| S - T \| < \eps$. Moreover,
\[
\| \widetilde{S} (0, x_2) \| = \| \widetilde{S} (0, x_2) - \widetilde{T} (0, x_2) \| \leq \| \widetilde{S} - \widetilde{T} \| < \eps
\]
for all $x_2 \in B_{X_2}$. We claim that $S \in \QNA (X_1 , Y)$.
As $\oplus_a$ is of type $1$, there is $K > 0$ such that $\|x_1 \| + K \|x_2 \| \leq \| (x_1, x_2)\|_a $ for every $(x_1, x_2) \in X$ (see \cite[Lemma~1.4]{CDJM}). Passing to a subsequence, we may assume that $\| x_n^{(1)} \| \longrightarrow \lambda_1$ and $\| x_n^{(2)} \| \longrightarrow \lambda_2$. Then we have that $\lambda_1 + K \lambda_2 \leq 1$.
If $\lambda_2 > 0$, passing to a subsequence again, we may assume that $\| x_n^{(2)} \| > 0$ for each $n \in \N$. Now,
\begin{align*}
\| u \| = \lim_n \| \widetilde{S} (x_n^{(1)}, x_n^{(2)}) \| &\leq \lim_n \|\widetilde{S} (x_n^{(1)}, 0 ) \| + \|x_n^{(2)} \| \left\| \widetilde{S} \left(0, \frac{x_n^{(2)}}{\|x_n^{(2)}\|}\right) \right\| \\
&\leq \lim_n (\|u\| \| x_n^{(1)} \| + \| u \|\| x_n^{(2)} \| \,\eps ).
\end{align*}
If we choose $\eps > 0$ to be smaller than $K >0$, we have
\[
1 \leq \lambda_1 + \eps \lambda_2 < \lambda_1 + K \lambda_2 \leq 1,
\]
which is a contradiction. This implies that $\lambda_2 = 0$ and
\[
\| S x_n^{(1)} - u \| \leq \| \widetilde{S} (x_n^{(1)}, 0) - \widetilde{S} (x_n^{(1)}, x_n^{(2)} ) \| +  \| \widetilde{S} (x_n^{(1)}, x_n^{(2)} ) - u \| \leq \| \widetilde{S}\| \|x_n^{(2)} \| + \| \widetilde{S} x_n - u \| \longrightarrow 0.
\]
This shows that $S\in \QNA(X_1,Y)$, finishing the proof of (a).

(b) is a particular case of (a) as $L$-summands are of type $1$.

(c) Put $Y = Y_1 \oplus_a Y_2$, and let $\eps >0$ and $T \in \Lin (X, Y_1)$ be given. Define $\widetilde{T} \in \Lin (X,Y)$ by $\widetilde{T} (x) := (T x, 0 )$ for every $x \in X$. Then $\|\widetilde{T}\| = \|T\|$ and there exists $\widetilde{S} \in \QNA (X, Y)$ such that $\| \widetilde{S} \| = \| \widetilde{T}\|$ and $\| \widetilde{S} - \widetilde{T} \| < \eps$. If we write $\widetilde{S} = (\widetilde{S}_1, \widetilde{S}_2)$ where $\widetilde{S}_j \in \Lin (X, Y_j)$ for $j =1, 2$, then
\begin{align*}
\| (\widetilde{S}_1x - T x, \widetilde{S}_2 x ) \|_{\infty} \leq \|\widetilde{S} x - \widetilde{T} x \|_a \leq \|\widetilde{S}- \widetilde{T} \| < \eps
\end{align*}
for all $x \in B_X$. It follows that $\|\widetilde{S}_1 - T \| < \eps$ and $\| \widetilde{S}_2 \| < \eps$. Choose a sequence $(x_n) \subseteq S_X$ such that $$\widetilde{S} x_n \longrightarrow u = (u_1, u_2) \in Y \quad \text{ with } \quad \|u \|_a = \|\widetilde{S}\|.$$ This implies that $\widetilde{S}_1 x_n \longrightarrow u_1$ and $\widetilde{S}_2 x_n \longrightarrow u_2$. Notice from $\|u_2 \| < \eps$ that $\|u_1\| > \| T \| -\eps$.
Let $y^* =(y_1^*, y_2^*) \in Y^*$ such that $\|y^*\|_{a^*} = 1$ and $y^* (u) = y_1^* (u_1) + y_2^* (u_2) = \| u\|_a$. It is easy to deduce that $ y_1^* (u_1) = \| y_1^*\| \|u_1\|$ and $y_2^* (u_2) = \|y_2^*\| \|u_2\|$. Define $S \in \Lin (X, Y_1)$ by
\[
S(x) := \|y_1^* \| \widetilde{S}_1 x + y_2^* (\widetilde{S}_2 x ) \frac{u_1}{\|u_1\|} \qquad \text{for } x \in X.
\]
Then, we have that
\begin{align*}
\|Sx \| \leq \|y_1^* \| \| \widetilde{S}_1 x \| + \| y_2^* \| \| \widetilde{S}_2 x \| \leq \|\widetilde{S} x \|_a \| y^* \|_{a^*} = \|\widetilde{S} x\|;
\end{align*}
hence $\|S \| \leq \| \widetilde{S} \|$. Note that
\[
S x_n = \|y_1^* \| \widetilde{S}_1 x_n + y_2^* (\widetilde{S}_2 x_n ) \frac{u_1}{\|u_1\|} \longrightarrow \|y_1^* \| u_1 + y_2^* (u_2) \frac{u_1}{\|u_1\|} = \frac{\|u\|_a}{\|u_1\|} u_1.
\]
Thus, $S \in \QNA(X, Y_1)$. To see that $S$ is close enough to $T$, observe first that $\| y_1^* \| \|u\|_a \geq \| y_1^* \| \|u_1 \| > \| u \|_a - \eps$. Hence for every $x \in B_X$,
\begin{align*}
\| Sx - Tx \| &\leq \| \|y_1^* \| \widetilde{S}_1 x - Tx \| + \| y_2^* (\widetilde{S}_2 x ) \| \\
&< (1 - \|y_1^* \| ) \|\widetilde{S} \| + \| \widetilde{S}_1 - T \| + \eps \\
&\leq \frac{\eps}{\| u \|_a } \|\widetilde{S} \| + 2\eps = 3 \eps.
\end{align*}
So, $\| S - T \| \leq 3 \eps$.
\end{proof}

A similar result to the previous one is the following one which borrows ideas from \cite[Proposition 2.8]{ACKLM}.

\begin{prop}
Let $X$ and $Y$ be Banach spaces and $K$ be a compact Hausdorff space. If $\QNA(X,C(K,Y))$ is dense in $\Lin(X,C(K,Y))$, then $\QNA(X,Y)$ is dense in $\Lin(X,Y)$.
\end{prop}

\begin{proof}
Let $\eps >0$ and $T \in \Lin (X, Y)$ be given. Define $\widetilde{T} \in \Lin (X , C(K, Y))$ as $(\widetilde{T} x ) (t) := Tx$ for every $x \in X$ and $t \in K$. It is clear that $\|\widetilde{T} \| = \|T \|$. Let $\widetilde{S} \in \QNA (X ,C(K, Y))$ be such that $\| \widetilde{S} \| = \| \widetilde{T} \|$ and $\| \widetilde{S} - \widetilde{T} \| < \eps$. Let $(x_n) \subseteq S_X$ be a sequence such that $\widetilde{S} x_n \longrightarrow f \in C(K, Y)$ with $\|\widetilde{S} \| = \| f \|$. Let $t_0 \in K$ so that $\| f (t_0) \| = \| f \|$, then $[\widetilde{S} x_n ](t_0) \longrightarrow f(t_0) \in \|f\| S_Y$. Define $S \in \Lin (X, Y)$ as $S(x) := [\widetilde{S} x ] (t_0)$ for every $x \in X$, then $\| S \| \leq \| \widetilde{S}\|$ and $S x_n = [\widetilde{S} x_n ] (t_0) \longrightarrow f(t_0)$. It follows that $S \in \QNA (X, Y)$. Note that
\begin{align*}
\| Sx - Tx \| = \bigl\| [\widetilde{S} x ] (t_0) - [\widetilde{T} x ] (t_0) \bigr\| \leq \| \widetilde{S} x - \widetilde{T} x \| < \eps
\end{align*}
for every $x \in B_X$; hence $\| S - T \| < \eps$.
\end{proof}

The third result of the section is that the denseness is preserved under $\ell_1$-sums of the domain space. Given a family $\{Z_i\colon i \in I\}$ of Banach spaces, we denote by $\left[\bigoplus_{i \in I} Z_i\right]_{\ell_1}$ the $\ell_1$-sum of the family.

\begin{cor}\label{coro:stability-ell1-ellinfty}
	Let $\{X_i\colon i \in I\}$ be a family of Banach spaces, let $X$ be the $\ell_1$-sum of $\{ X_i\}$, and $Y$ be a Banach space. Then, $\QNA(X,Y)$ is dense in $\Lin(X,Y)$ if and only if $\QNA(X_i,Y)$ is dense in $\Lin(X_i,Y)$ for every $i \in I$.
\end{cor}

The proof of the ``if part'' is based on the corresponding one given in \cite{PS} for norm attaining operators.

\begin{proof}
As each $X_i$ is an $L$-summand of $X$, it follows from Proposition~\ref{prop:main_stability_result}.(b) that they inherit the property from $X$. Conversely, let $\eps >0$ and $T \in \Lin (X, Y)$ with $\|T \| =1$ be given. As $\|T \| = \sup\{\| T E_i \| : i \in I\}$ where $E_i$ is the natural isometric inclusion from $X_i$ into $X$, we may choose $i_0 \in I$ such that $\| T E_{i_0} \| > 1 - \eps$. Choose $S_{i_0} \in \QNA (X_{i_0}, Y)$ such that $\| S_{i_0}\| = 1$ and $\| S_{i_0} - T E_{i_0} \| < \eps$. Let $(x_n) \subseteq S_{X_{i_0}}$ be a sequence such that $S_{i_0} x_n \longrightarrow u$ for some $u \in Y$ with $\| u \| = 1$. Consider the operator $S \in \Lin (X, Y)$ so that $S E_{i_0} = S_{i_0}$ and $S E_j = T E_j$ for every $j \neq i_0$. Then $\| S \| \leq 1$ and $\| S - T \| = \| S_{i_0} - TE_{i_0} \| < \eps$.
Notice that $S (E_{i_0} x_n) = S_{i_0} x_n \longrightarrow u$, thus $S \in \QNA (X ,Y)$.
\end{proof}

In the aforementioned paper \cite{PS} it is shown an analogous result to the above one for the denseness of norm attaining operator for $c_0$- or $\ell_\infty$-sums of range spaces. We do now know whether such result has a version for quasi norm attainment. Actually, we do not know whether the fact that the denseness of quasi norm attaining operators for the pairs $(X,Y_1)$ and $(X,Y_2)$ implies the denseness of $\QNA(X,Y_1\oplus_\infty Y_2)$, see Problem~\ref{problem:ellinftysum} below.

\section{Remarks and open questions}\label{section:Remarks-open-questions}
Our final aim in the paper is to present some open problems and remarks on quasi norm attaining operators.

\subsection{Extensions of results on norm attaining operators} We would like to study whether some results valid for norm attaining operators remain true for quasi norm attaining operators. First, it would be of interest whether some more negative results on the denseness of norm attaining operators actually provide negative examples on the denseness of quasi norm attaining operators or not. For instance, the following questions can be of interest.

\begin{problem}
Is $\QNA(L_1[0,1],C[0,1])$ dense in $\Lin(L_1[0,1],C[0,1])$?
\end{problem}

Observe that it is shown in \cite{Schachermayer-classical} that $\NA(L_1[0,1],C[0,1])$ is not dense in $\Lin(L_1[0,1],C[0,1])$.

\begin{problem}
Let $Y$ be a strictly convex Banach space. Is it true that $Y$ has the RNP if (and only if) $\QNA(L_1[0,1],Y)$ is dense in $\Lin(L_1[0,1],Y)$?
\end{problem}

It is shown in \cite{Uhl} that the analogous result for norm attaining operators is true.

On the other hand, a couple of questions which have been stated along the paper can be also included in this subsection as they are related to results for the denseness of norm attaining operators.

\begin{problem}\label{problem:ellinftysum}
Let $X$, $Y_1$, $Y_2$ be Banach spaces such that $\QNA(X,Y_j)$ is dense in $\Lin(X,Y_j)$ for $j=1,2$. Is $\QNA(X,Y_1\oplus_\infty Y_2)$ dense in $\Lin(X,Y_1\oplus_\infty Y_2)$?
\end{problem}

The positive answer to this question for norm attaining operators was given in \cite{PS}.

Let us comment that a positive answer to Problem~\ref{problem:ellinftysum} would give an example of a Banach space $Z$ such that $\QNA(Z,Z)$ is dense in $\Lin(Z,Z)$ while $\NA(Z,Z)$ is not dense (indeed, $Z=G\oplus_\infty \ell_2$ where $G$ is Gowers' space describe in Example~\ref{example:Gowers-injective-QNA-no-NA} would work). We do not know whether such an example exists.

\begin{problem}\label{problem:ZZ-QNA-not-NA}
Does there exist a Banach space $Z$ such that $\QNA(Z,Z)$ is dense in $\Lin(Z,Z)$ while $\NA(Z,Z)$ is not dense?
\end{problem}

\vspace{1ex}

\subsection{Lindenstrauss properties} It would be of interest to study the version for quasi norm attainment of Lindenstrauss properties A and B. Let us say that a Banach space $X$ has \emph{property quasi A} if $\overline{\QNA(X,Z)}=\Lin(X,Z)$ for every Banach space $Z$; a Banach space $Y$ has \emph{property quasi B} if $\overline{\QNA(W,Y)}=\Lin(W,Y)$ for every Banach space $W$.

A list of some known results that we may write down on these properties, using both previously known results and results from this paper, is the following.

\begin{itemize}
\item[\textup{(a)}] $X$ has property quasi A in every equivalent norm if and only if $X$ has the RNP;
\item[\textup{(b)}] Separable Banach spaces (actually, spaces admitting a long biorthogonal system) can be equivalently renormed to have property A, and so property quasi A.
\item[\textup{(c)}] $Y$ has property quasi B in every equivalent norm if and only if $Y$ has the RNP;
\item[\textup{(d)}] Every Banach space can be equivalently renormed to have property B, and so property quasi B.
\end{itemize}

Assertion (a) and (c) follows from our Corollary~\ref{corollary:RNP-equivalence}; (b) and (d) appear in \cite{Godun-Troyanski} and in \cite{Partington}, respectively.

Therefore, the following question seems to be open.

\begin{problem}
Is it possible for every Banach space to be renormed equivalently to have property quasi A?
\end{problem}

The study of Lindenstrauss properties A and B provided many interesting results on the geometry of the involved Banach spaces, and the same can be true for the new analogous properties. For instance, the following result is an extension of a result by J.~Lindenstrauss \cite{Lindenstrauss} to the case of quasi norm attaining operators. Recall that a Banach space is said to be \emph{locally uniformly rotund} (\emph{LUR} in short) if for all $x,x_n\in B_X$ satisfying $\lim_n \|x_n+x\|=2$ we have $\lim_n \|x_n-x\|=0$. Separable Banach spaces and reflexive ones admit LUR norms. We refer the reader to \cite[Chapter~7]{FHHMZ} for background.

\begin{prop}
Let $X$ be a Banach space with property quasi A.
\begin{enumerate}
\item[\textup{(a)}] If $X$ is isomorphic to a strictly convex space, then $B_X$ is the closed convex hull of its exposed points.
\item[\textup{(b)}] If $X$ is isomorphic to a locally uniformly rotund space, then $B_X$ is the closed convex hull of its strongly exposed points.
\end{enumerate}
\end{prop}

The proofs of (a) and (b) are very similar and are based on the corresponding proofs given in \cite[Theorem~2]{Lindenstrauss}, so we only leave here the idea of the proof of (b). Indeed, it is shown in the proof of \cite[Theorem~2]{Lindenstrauss} that for a Banach space $X$ which is isomorphic to a LUR space, if $B_X$ is not the closed convex hull of its strongly exposed points, then there exist a Banach space $Y$ and a monomorphism $T\colon X \longrightarrow Y$ such that $T \notin \overline{\NA (X,Y)}$. Combining this result with Lemma~\ref{lemma:paya-denseness}, we have $T \notin \overline{\QNA(X,Y)}$ and so $X$ fails property quasi A.

It would be interest to find other necessary conditions for properties quasi A and quasi B. For instance, there is a necessary condition for Lindenstrauss property B given in \cite[Theorem~3]{Lindenstrauss} in terms of smooth points which we do not know whether it is still valid for quasi norm attaining operators.

Let us also mention that while we know that Lindenstrauss property B is not the same that property quasi B (for instance, $Y=\ell_2$ has property quasi B as it is reflexive but not Lindenstrauss property B, see Example~\ref{examples:QNAdense-NAnot}.(a)), we do not know of any example of Banach space having property quasi A without having Lindenstrauss property A.

\begin{problem}
Does property quasi A imply Lindenstrauss property A?
\end{problem}

\vspace{1ex}

\subsection{Uniquely quasi norm attaining operators} We would like now to discuss the relation between uniquely quasi norm attaining operators and quasi norm attaining operators. It was already commented that both concepts are different: the identity in a Banach space of dimension greater than one is clearly quasi norm attaining but not uniquely. Aiming at the denseness, as a consequence of the results in Section~\ref{section:RNP-sufficient}, if $X$ or $Y$ has the RNP, then uniquely quasi norm attaining operators from $X$ to $Y$ are dense (see Corollary~\ref{coro:XRNPYRNP-QNAdense-strongversion}). So one may wonder whether the denseness of quasi norm attaining operators actually implies the stronger result of denseness of uniquely quasi norm attaining operators, but the following example shows that this is not the case, even if we have denseness of norm attaining operators.

\begin{example}
Let $\Id \in \Lin(c_0,c_0)$ be the identity map. Then, $\Id\in \NA(c_0,c_0)\subset \QNA(c_0,c_0)$, $\NA(c_0,c_0)$ is dense in $\Lin(c_0,c_0)$, but $\Id$ does not belongs to the closure of the set of uniquely quasi norm attaining operators.
\end{example}

\begin{proof}
It is clear that $\Id\in \NA(c_0,c_0)$ and the denseness of $\NA(c_0,c_0)$ follows from the fact that $c_0$ has property $\beta$ and we may use Examples~\ref{examples:NA-dense}. So it suffices to prove that $\Id$ cannot be approximated by uniquely quasi norm attaining operators. Indeed, suppose that there exists $T \in \Lin(c_0,c_0)$ which uniquely quasi norm attains its norm such that $\|T-\Id\|<\frac{1}{4}$. Consider $y_0 \in \|T\| S_{c_0}$ such that $T$ uniquely quasi attain its norm towards $y_0$ and take a sequence $(x_n) \subset S_{c_0}$ satisfying that
$$
Tx_n \longrightarrow y_0 \quad \text{and} \quad \|x_n-y_0\|<\frac{1}{2}.
$$
Let $m_0 \in \N$ be such that $|y_0(m_0)| < \frac{1}{4}$, and consider the sequence $\bigl( x_n + \frac{1}{4} \lambda_n e_{m_0} \bigr) \subset S_{c_0}$, where $\lambda_n \in \{-1,+1\}$ is chosen so that
$$
\left\| T \left( x_n + \frac{1}{4} \lambda_n e_{m_0} \right) \right\| \geq \|Tx_n\|
$$
for each $n \in \N$. This is possible by an easy convexity argument: if $$\left\| T \left( x_n + \frac{1}{4} e_{m_0} \right) \right\|<\|Tx_n\| \quad \text{and} \quad \left\| T \left( x_n - \frac{1}{4} e_{m_0} \right) \right\|<\|Tx_n\|$$ for some $n\in \N$, then
$$
2\|Tx_n\|\leq \left\| T \left( x_n + \frac{1}{4} e_{m_0} \right) \right\|\, + \, \left\| T \left( x_n - \frac{1}{4} e_{m_0} \right) \right\|<\|Tx_n\| + \|Tx_n\|,
$$
a contradiction. Now, we may assume by taking subsequence that $\lambda_n=\lambda_0$ for all $n \in \N$. Since $\|Tx_n\| \longrightarrow \|T\|$, we have that
$$
\left\| T\left( x_n + \frac{1}{4} \lambda_0 e_{m_0} \right) \right\| \longrightarrow \|T\|.
$$
Thus, as $T$ uniquely quasi attains its norm, there exist a subsequence $(x_{\sigma(n)})$ of $(x_n)$ and a scalar $\theta_0 \in \T$ such that
$$
T \left( x_{\sigma(n)} +\frac{1}{4} \lambda_0 e_{m_0} \right) \longrightarrow \theta_0 y_0,
$$
or equivalently,
$$
\frac{1}{4} \lambda_0 Te_{m_0} = (\theta_0-1) y_0.
$$
It follows that $|\theta_0-1| \leq \frac{1}{4}$ and hence that
\begin{equation}\label{eq:c0c0-UQNA-1}
\frac{1}{16} \geq \bigl|(\theta_0-1)y_0(m_0)\bigr| = \frac{1}{4}\bigl|[Te_{m_0}](m_0)\bigr|.
\end{equation}
On the other hand, writing as usual $e_{m_0}$ to denote the element of $c_0$ whose $m_0^{\,\text{th}}$ coordinate is $1$ and the rests are $0$, we have that
$$
\|T-\Id\|\geq \bigl\|[T-\Id](e_{m_0})\bigr\|=\|T(e_{m_0})-e_{m_0}\|\geq \bigl|[Te_{m_0}](m_0)-1\bigr|\geq 1 - \bigl|[Te_{m_0}](m_0)\bigr|,
$$
so
$$
\frac{1}{4}\bigl|[Te_{m_0}](m_0)\bigr|\geq \frac{1}{4} \bigl(1- \|T-\Id\| \bigr) > \frac{3}{16}.
$$
This contradicts \eqref{eq:c0c0-UQNA-1}, finishing the proof.
\end{proof}

The next result gives a positive condition to pass from the denseness of quasi norm attaining operators to uniquely quasi norm attaining operators: that the range space is locally uniformly rotund. Actually, in this case, we can get a result valid operator by operator.

\begin{prop}\label{prop:strictly-convex-strongly-expose}
Let $X$ and $Y$ be Banach spaces such that $Y$ is LUR. Then, every $T \in \QNA(X,Y)$ can be approximated by a uniquely quasi norm attaining operators.
\end{prop}

\begin{proof}
Let $T \in \QNA(X,Y)$ and $\eps>0$ be given. We may assume here that $0<\eps\leq\|T\|$. Let $(x_n) \subset S_X$ and $y_0 \in \|T\| S_Y$ be satisfying that $Tx_n \longrightarrow y_0$. Choose $y_0^* \in S_{Y^*}$ so that $y_0^*(y_0)=\|T\|$, and define an operator $S \in \Lin(X,Y)$ by
	\[
	S(x):= Tx + \eps y_0^*(Tx) \frac{y_0}{\|T\|^2}.
	\]
It is easy to see that $S \in \QNA(X,Y)$ with $Sx_n \longrightarrow y_0\Bigl(1+ \frac{\eps}{\|T\|}\Bigr)$ and $\|S\|=\|T\|+\eps$.
Suppose now that there is a sequence $(z_n) \subset B_X$ such that $\|Sz_n\| \longrightarrow \|S\|$. That is,
	\[
	\left\|Tz_n + \eps y_0^*(Tz_n) \frac{y_0}{\|T\|^2} \right\| \longrightarrow \|T\|+\eps \qquad \text{as } n \to \infty.
	\]
We first note here that $\|Tz_n\| \longrightarrow \|T\|$ as $n$ tends to $\infty$, otherwise it implies that $\|Sz_n\|$ does not converge to $\|S\|$, a contradiction. Take a subsequence $(z_{\sigma(n)})$ of $(z_n)$ such that $y_0^*(Tz_{\sigma(n)})$ is convergent. The fact that $\|Sz_{\sigma(n)}\| \longrightarrow \|S\|$ gives us that $\lambda_0 := \lim_n \frac{y_0^*(Tz_{\sigma(n)})}{\|T\|} \in \T$. Hence we have that
	\begin{equation}\label{eq:LUR-convergent}
	\left\|Tz_{\sigma(n)} + \eps \lambda_0 \frac{y_0}{\|T\|} \right\| \longrightarrow \|T\|+\eps \qquad \text{as } n \to \infty.
	\end{equation}
Now, we claim that $\|Tz_{\sigma(n)} + \lambda_0y_0\| \longrightarrow 2\|T\|$. If the claim holds, then by the local uniform rotundity of $Y$, we can conclude that $Tz_{\sigma(n)} \longrightarrow \lambda_0 y_0$ and thus that $Sz_{\sigma(n)} \longrightarrow \lambda_0 y_0 \left( 1 + \frac{\eps}{\|T\|} \right)$, finishing the proof.

Since it is clear that $\|Tz_{\sigma(n)} + \lambda_0 y_0\| \leq 2\|T\|$, it suffices to show the opposite inequality. By the triangular inequality, we have that
	\[
	\|Tz_{\sigma(n)} + \lambda_0y_0\| \geq \left\| Tz_{\sigma(n)} \frac{\|T\|}{\eps} + \lambda_0 y_0 \right\| - \left(\frac{\|T\|}{\eps} - 1\right)\|T\|.
	\]
Observe that \eqref{eq:LUR-convergent} yields
	\[
	\left\| Tz_{\sigma(n)} \frac{\|T\|}{\eps} + \lambda_0 y_0 \right\| \longrightarrow (\|T\| +\eps) \frac{\|T\|}{\eps} \qquad \text{as } n \to \infty.
	\]
We then obtain that $\lim \|Tz_{\sigma(n)} + \lambda_0y_0\| \geq (\|T\|+\eps) \frac{\|T\|}{\eps} - \left(\frac{\|T\|}{\eps} - 1 \right)\|T\| = 2\|T\|$.
\end{proof}

It would be interesting to study more results analogous to the previous one.

\begin{problem}
Find other sufficient conditions allowing us to approximate quasi norm attaining operators by uniquely quasi norm attaining operators.
\end{problem}

\vspace{1ex}

\subsection{Quasi norm attaining endomorphisms} M.~I.~Ostroskii asked in \cite[p.~65]{Maslyu-Plichko} whether there exists an infinite dimensional Banach space such that $\NA(X,X)=\Lin(X,X)$, give some remarks on the possible example, and shows that the only possible candidates for $X$ are separable reflexive spaces without $1$-complemented infinite-dimensional subspaces having the approximation property. There is some more information in the web page\newline \href{https://mathoverflow.net/questions/232291/}{https://mathoverflow.net/questions/232291/}.

The version of the problem for quasi norm attaining operators could also be of interest.

\begin{problem}
Is there any infinite dimensional Banach space $X$ such that $\QNA(X,X)=\Lin(X,X)$?
\end{problem}

Some observations on the problem:
\begin{itemize}
\item If $X$ is reflexive, then $\QNA(X,X)=\NA(X,X)$ (by Proposition~\ref{prop:reflexive}) and so the problem is the same as Ostroskii's problem.
\item As $\K(X,X)\subset \QNA(X,X)$, one may think that the answer can be found among those Banach spaces with very few operators, that is, those $X$ such that $\Lin(X,X)=\{\lambda \Id + S\colon \lambda\in \KK,\, S\in \K(X,X)\}$ see \cite{ArgyrosHaydon-Acta,ArgyrosMotakis-TAMS} for a reference on this. But, again, in this case ``most'' quasi norm attaining operators are actually norm attaining, as the following easy result shows.
\end{itemize}

\begin{remark}
{\slshape Let $X$ be a Banach space, $\lambda\in \KK\setminus \{0\}$, $S\in \K(X,X)$, and write $T:=\lambda\Id + S$. If $T\in \QNA(X,X)$, then $T\in \NA(X,X)$.}\newline Indeed, take $(x_n)$ in $B_X$ such that $Tx_n\longrightarrow u\in \|T\|S_X$ and, by compactness, consider a subsequence $(x_{\sigma(n)})$ of $(x_n)$ such that $Sx_{\sigma(n)}\longrightarrow z\in X$. Now, $$x_{\sigma(n)}\longrightarrow \lambda^{-1}\bigl(u-z\bigr)=: x_0$$ and we have that $x_0\in B_X$ and $Tx_0=u$, so $\|Tx_0\|=\|T\|$ and $T\in \NA(X,X)$.
\end{remark}

\vspace*{0.5cm}

\noindent \textbf{Acknowledgment:\ } The authors thank Rafael Pay\'{a} for many conversations on the topic of this manuscript and, in particular, for providing the idea of Lemma~\ref{lem_paya}. They also thank D.~Werner
for kindly answering several inquires related to the topics of this manuscript and providing valuable references.

\end{document}